\documentclass[12pt]{amsart}
\usepackage{BAstyle}

\usepackage[margin=1in]{geometry}
\usepackage{mathrsfs}
\usepackage{enumerate}

\newtheorem*{notation*}{Notation}

\title[A Random Group with Local Data]{A Random Group with Local Data\\Realizing Heuristics for Number Field Counting}
\author{Brandon Alberts}

\begin{document}
\maketitle

\begin{abstract}
We define a group with local data over a number field $K$ as a group $G$ together with homomorphisms from decomposition groups $\Gal(\overline{K}_p/K_p)\to G$. Such groups resemble Galois groups, just without global information. Motivated by the use of random groups in the study of class group statistics, we use the tools given by Sawin--Wood to construct a random group with local data over $K$ as a model for the absolute Galois group $\Gal(\overline{K}/K)$ for which representatives of Frobenius are distributed Haar randomly as suggested by Chebotarev density. We utilize Law of Large Numbers results for categories proven by the author to show that this is a random group version of the Malle-Bhargava principle. In particular, it satisfies number field counting conjectures such as Malle's Conjecture under certain notions of probabilistic convergence including convergence in expectation, convergence in probability, and almost sure convergence. These results produce new heuristic justifications for number field counting conjectures, and begin bridging the theoretical gap between heuristics for number field counting and class group statistics.
\end{abstract}



\section{Introduction}

Class groups and Galois groups of unramified extensions of number fields $K$ are predicted to be distributed along families of number fields according to certain random groups; that is, there exists a probability measure $\mu_{\mathcal{F},\mathscr{C}}$ on the space of profinite groups such that conjecturally
\begin{align}\label{eq:classgroup}
\lim_{X\to\infty}\frac{\ds \#\{K\in \mathcal{F} : \mathscr{C}(K)\cong G,\ \disc(K)\le X\}}{\ds \#\{K\in \mathcal{F} : \disc(K)\le X\}} = \mu_{\mathcal{F},\mathscr{C}}(G),
\end{align}
where $\mathcal{F}$ is a family of number fields and $\mathscr{C}$ could be the class group, $\Gal(K^{un}/K)$, or a similar construction such as the Galois group of the maximal unramified prime to $2|\Gal(K/\Q)|$ extension of $K$ \cite{friedman-washington1989,boston-bush-hajir2017,liu-wood-zureick-brown2019}. The classical version of this principle is the Cohen--Lenstra heuristics \cite{cohen-lenstra1984}, which are shown by Friedman--Washington to be equivalent to predicting that the $p$-parts of class groups of quadratic fields are distributed as the cokernels of certain random $p$-adic matrices \cite{friedman-washington1989}. These heuristics have lead to a greater understanding of the structure of unramified extensions and highlight interesting equidistribution properties in the absolute Galois group.

Malle's conjecture and generalizations for the distributions of number fields are closely related to the distributions of unramified extensions, as they ask more general questions about the rate of growth of functions like
\begin{align}\label{eq:counting}
\#\{K\in\mathcal{F} : \text{other conditions, }\disc(K)\le X\}
\end{align}
as $X\to\infty$. The classical example studied by Malle \cite{malle2002,malle2004} is
\[
N(K,G;X):=\#\{L/K : [L:K]=n,\ \Gal(L/K)\cong G,\ {\rm Nm}_{L/K}\disc(L/K)\le X\}.
\]
for $G\subset S_n$ a transitive subgroup and $\Gal(L/K)\subset S_n$ the Galois group of the Galois closure $\widetilde{L}/K$ together with the action on the $n$ embeddings $L\hookrightarrow \widetilde{L}$. Conjectural rates of growth for these counting functions are made by appealing to local information and presuming an ``average local-to-global principle". Despite the clear similarities between Malle's counting function and the counting functions appearing in class group statistics, no association between Malle's conjecture and random groups currently exists in the literature. The goal of this paper is to bridge the gap between the theory of distributions of unramified extensions and distributions of other families of number fields by constructing a ``random object model" for the absolute Galois group $G_K:=\Gal(\overline{K}/K)$, which can be used to witness Malle's conjecture.

Towards this end, we define two categories on which we will build a random object modeling the absolute Galois group. Let $G_{K_p} := \Gal(\overline{K}_p/K_p)$ be the decomposition group at the place $p$ of $K$.
\begin{enumerate}
\item the category of groups with local data, ${\rm proGrp}(K)$, whose objects are pairs $(G,\phi)$ of a profinite group and a tuple $\phi=(\phi_p)$ of continuous homomorphisms $\phi_p:G_{K_p}\to G$ for each place $p$ of $K$. See Definition \ref{def:proGrp} for the full definition, including morphisms.
\item the category of finite groups with finite local data ${\rm Grp}(K)$ whose objects are triples $(G,S,\phi)$ of a finite group $G$, a finite set of places $S$ of $K$, and $\phi = (\phi_p)_{p\in S}$ a family of continuous homomorphisms $\phi_p:G_{K_p}\to G$ for each place $p\in S$. See Definition \ref{def:Grp} for the full definition, including morphisms.
\end{enumerate}

We will prove that ${\rm proGrp}(K)$ is (up to a null set) isomorphic to the category of pro-objects of ${\rm Grp}(K)$. The moments problem over categories of pro-objects has been solved in a wide class of cases by recent work of Sawin--Wood \cite{sawin-wood2022} for very general sequences of finite moments, subject to some mild conditions. We will prove that ${\rm Grp}(K)$ satisfies these conditions and give a family of well-behaved sequences of finite moments in the sense defined by Sawin--Wood, see Proposition \ref{prop:wb}.

In particular, using the main results of \cite{sawin-wood2022} we prove the following:
\begin{theorem}\label{thm:muK_construct}
Let $K$ be a number field. Then there exists a unique probability measure $\mu_K^{\rm MB}$ on the isomorphism classes of ${\rm proGrp}(K)$ such that
\[
\int_{{\rm proGrp}(K)} \#{\rm Epi}(\mathscr{G},(G,S,\phi))\ d\mu_K^{\rm MB}(\mathscr{G}) = |G^{\rm ab}[|\mu(K)|]|^{-1}|G|^{-|S\cup P_\infty| + 1}
\]
for each $(G,S,\phi)\in {\rm Grp}(K)$. Here, $\mu(K)$ is the group of roots of unity in $K$ and $P_\infty$ is the set of infinite places of $K$.
\end{theorem}

The superscript ``MB" stands for ``Malle-Bhargava", as we will show that $\mu_K^{\rm MB}$ is, in some sense, a random group analog to the Malle-Bhargava principle \cite{bhargava2007,wood2019}. The finite moments of $\mu_K^{\rm MB}$ are constructed from the Chebotarev density theorem and heuristic predictions for class group statistics. We will show in Lemma \ref{lem:originalmoment} that these finite moments agree with those predicted by the Malle-Bhargava local series.

Malle's counting function is, up to the Galois correspondence, counting surjections from the absolute Galois group to $G$ with bounded discriminant. Given a group with local data $\mathscr{G}$, we can define the discriminant of a surjection $\pi:\mathscr{G}\to G$ via the local discriminants $\disc(\pi|_{G_{K_p}})$. Thus, Malle's counting function can be extended to groups with local data. The author recently proved a version of the Law of Large Numbers for counting functions on random objects in a category \cite{alberts2022}, showing that these functions often have a particular growth rate with probability $1$. We convert the discriminant ordering for number fields to this context, and using the main results of \cite{alberts2022} we prove that $\mu_K^{\rm MB}$ satisfies Malle's conjecture \emph{with a leading constant given as a convergent Euler product} in the n\"aively expected cases with probability $1$.

\begin{theorem}\label{thm:prob1intro}
Let $K$ be a number field, $G\subset S_n$ be a transitive group, and $\mu_K^{\rm MB}$ the constructed distribution of groups with local data. With $\mathscr{G}$ distributed according to $\mu_K^{\rm MB}$, it follows that
\begin{itemize}
\item[(i)] For any $\epsilon > 0$,
\[
\frac{\#\{\pi\in \Surj(\mathscr{G},G): |\disc(\pi)| \le X\}}{X^{1/a(G)+\epsilon}} \overset{a.s.}{\longrightarrow} 0
\]
as $X\to \infty$, where the ``a.s." stands for ``converges almost surely".
\item[(ii)] If $G = \langle g\in G : \ind(g) = a(G)\rangle$ is generated by minimal index elements then
\[
\frac{\#\{\pi\in \Surj(\mathscr{G},G): |\disc(\pi)| \le X\}}{c(K,G)X^{1/a(G)}(\log X)^{b(K,G) - 1}} \overset{p.}{\longrightarrow} 1
\]
as $X\to \infty$, where the ``p." stands for ``converges in probability".
\item[(iii)] If every proper normal subgroup $N\normal G$ satisfies one of
\begin{itemize}
\item[(a)] $N$ contains no minimal index elements, or
\item[(b)] $G\setminus N$ contains at least two $K$-conjugacy classes of minimal index,
\end{itemize}
then
\[
\frac{\#\{\pi\in \Surj(\mathscr{G},G): |\disc(\pi)| \le X\}}{c(K,G)X^{1/a(G)}(\log X)^{b(K,G) - 1}} \overset{a.s.}{\longrightarrow} 1
\]
as $X\to \infty$, where the ``a.s." stands for ``converges almost surely".
\end{itemize}
Here $a(G)$, $b(K,G)$, and $K$-conjugacy classes are defined as in Malle's conjecture (Conjecture \ref{conj:malle}), and
\begin{align*}
c(K,G) = &\frac{(\Res_{s=1}\zeta_K(s))^{b(K,G)}}{a(G)^{b(K,G)-1}(b(K,G)-1)!|G^{\rm ab}[|\mu(K)|]|\cdot|G|^{u_K}} \prod_{p\mid \infty}\left(\sum_{f\in \Hom(G_{K_p},G)} 1\right)\\
&\cdot\prod_{p\nmid \infty} \left[\left(1 - p^{-1}\right)^{b(K,G)}\left(\frac{1}{|G|}\sum_{f\in \Hom(G_{K_p},G)} p^{-\frac{\nu_p\disc(f)}{a(G)}}\right)\right]
\end{align*}
for $u_K = \rk \mathcal{O}_K^{\times}$ the unit rank of $K$.
\end{theorem}

These methods are very robust, and can be applied to a number of generalizations of Malle's counting function. For the sake of clarity, we leave the most general version of this statement (such as restricting local conditions at infinitely many places) for a future paper. However, in the course of our proof we will require a more general version of Theorem \ref{thm:prob1intro}. This is due to natural relationships between discriminant orderings and non-discriminant orderings that we take advantage of in the proof. This more general result is stated in Theorem \ref{thm:muK_Malle_general} and includes, in particular, the product of ramified primes ordering.

Theorem \ref{thm:prob1intro} also highlights the fact that $\mu_K^{\rm MB}$ is a random group analog of the Malle-Bhargava principle and Malle's original conjecture - as it agrees with Malle's conjecture even in cases where Malle's conjecture is \emph{wrong}. Structuring these predictions as a random group will help us to highlight what is going wrong with the Malle-Bhargava principle and look for ways to fix it. In Section \ref{sec:knownissues}, we disucss the known obstructions to Malle's conjecture and how they interact with the random group with local data $\mu_K^{\rm MB}$. We refrain from making conjectures, but instead focus on how phrasing Malle's prediction in terms of a random group with local data clarifies these obstructions and gives an indication of how to produce improved predictions.

\subsection{Historical Background and Motivation}

For $K$ a number field and $G\subseteq S_n$ a transitive subgroup, Malle's conjecture can be rephrased via the Galois correspondence to be about counting surjections from the absolute Galois group
\[
\Surj(G_K,G;X) = \{\pi:G_K\twoheadrightarrow G : {\rm Nm}_{K/\Q}\disc(\pi)\le X\},
\]
where $G_K$ denotes the absolute Galois group of $K$ and $\disc(\pi)$ is the discriminant of the field fixed by $\pi^{-1}(\Stab_G(1))$. The Galois correspondence implies $\#\Surj(G_K,G;X) = |\Aut(G)|\cdot N(K,G;X)$, so determining the rate of growth of this function is an equivalent problem.

\begin{conjecture}[Malle \cite{malle2002,malle2004}]\label{conj:malle}
Let $G\subset S_n$ be a finite transitive group. Let $a(G) = \min_{g\ne 1} \ind(g)$, where the index of an element is given by $\ind(g) = n - \#\{\text{orbits of }g\}$. Then
\begin{itemize}
\item[(i)] (Strong form) Let $\chi:G_K\to \hat{\Z}^\times$ act on $G$ by $x.g = g^{\chi(x)}$, and let $b(K,G)$ be the number of orbits under the cyclotomic action of conjugacy classes $c\subset G$ for which $\ind(c) = a(G)$ is minimal. Then there exists a positive constant $c(K,G)$ for which
\[
\#\Surj(G_K,G;X) \sim c(K,G) X^{1/a(G)}(\log X)^{b(K,G)-1}
\]
as $X\to \infty$.
\item[(ii)] (Weak form)
\[
X^{1/a(G)} \ll \#\Surj(G_K,G;X) \ll_{\epsilon} X^{1/a(G)+\epsilon}
\]
as $X\to \infty$.
\end{itemize}
\end{conjecture}

The strong form is known to be false in some cases, in particular $C_3\wr C_2\subseteq S_6$ as shown by Kl\"uners \cite{kluners2005}. Yet, it is known to be true in many other cases including
\begin{itemize}
\item abelian groups \cite{maki1985,wright1989},
\item $S_3$ in degree $3$ \cite{davenport-heilbronn1971,datskovsky-wright1988} and degree $6$ \cite{bhargava-wood2007},
\item $S_4$ and $S_5$ in degree $4$ and $5$ respectively \cite{bhargava2014,bhargava-shankar-wang2015},
\item $D_4$ in degree $4$ \cite{cohen-diaz-y-diaz-olivier2002},
\item generalized quaternion groups \cite{klunersHab2005},
\item most wreath products $C_2\wr H$ \cite{kluners2012},
\item $A\times S_n$ for $n=3,4,5$ and $A$ an abelian group without certain small prime divisors \cite{jwang2021,masri-thorne-tsai-jwang2020}, and
\item ${\rm Heis}_3\subseteq S_9$ with $K = \Q$ \cite{fouvry-koymans2021}.
\end{itemize}
The value of the constant $c(K,G)$ is also the subject of investigation, but much less is known about what value to expect here. Bhargava originally formulated the Malle-Bhargava principle in part to predict the value of this constant when $G = S_n$ is the symmetric group \cite{bhargava2007}. See \cite{wood2009} for a broad investigation in the case that $G$ is abelian.

Theorem \ref{thm:prob1intro} states that a group with local data $\mathscr{G}$ distributed according to $\mu_K^{\rm MB}$ satisfies Malle's conjecture for $G$-extensions with probability $1$ (under some mild conditions on $G$). The category of groups with local data is built to resemble the absolute Galois group $G_K$, being profinite groups with decomposition subgroups. In fact $\mu_K^{\rm MB}$ was built out of properties of the absolute Galois group like Chebotarev density. It stands to reason that we could heuristically infer information about $G_K$ from information about probability $1$ events in $\mu_K^{\rm MB}$, even though $G_K$ is a deterministic object.

Using a random model for a deterministic object has precedence, notably with the Cram\'er random model for the set of prime numbers \cite{cramer1994,granville1995}. With such models, behavior that occurs 100\% of the time is said to provide evidence that we should expect the same behavior for the corresponding deterministic object. For prime numbers, such random models are used to justify predictions like the Hardy--Littlewood conjecture, Goldbach's conjecture, and many other conjectures involving prime gaps.

Along this line of thinking, Theorem \ref{thm:prob1intro} gives behavior for groups with local data with probability $1$ which can be considered good evidence for $G_K$ to share those properties. This form of justification is stated as a Vast Counting Heuristic in \cite[Heuristic 1.7]{alberts2022}, where we can make predictions if we expect $G_K$ to be ``typical" among groups with local data distributed according to $\mu_K^{\rm MB}$. Of course, it is well known that Malle's conjecture is false as stated - Kl\"uners provided the first counter example in $C_3\wr C_2\subseteq S_6$ for which Malle's predicted $b$-invariant is too small \cite{kluners2005}. Kl\"uners' counter example is witnessing some atypical behavior for $G_\Q$ among groups with local data distributed according to $\mu_K^{\rm MB}$, specifically the behavior that $\Gal(\Q(\zeta_3)/\Q)$ is a quotient of $G_\Q$.

For this reason, we do not attempt to make any conjectures in this paper. Our intention is to get the ball rolling on modeling counting functions in the style of Malle's conjecture with random objects, but we recognize that the absolute Galois group is known to have some atypical behaviors. At the end of the paper, we include a discussion of how Theorem \ref{thm:prob1intro} compares to the known cases of (and counter examples for) Malle's conjecture. We do not attempt to solve these problems in this paper, but rather focus on explaining how to interpret these issues in the world of random groups with local data.  It will be a goal of future work to use this framework to better capture behaviors of the absolute Galois group and create concrete predictions that are more accurate than the Malle-Bhargava principle.

\subsection{Layout of the Paper}

In Section \ref{sec:localdata} we construct the category of groups with local data and show that this is the category of pro-objects of a diamond category in the sense of \cite{sawin-wood2022}. Additionally, we prove Proposition \ref{prop:wb} giving a family of ``well-behaved sequences" as defined in \cite{sawin-wood2022}. These results prepare the category for solving the moment problem to construct a probability measure using to construct a probability measure with given finite moments using the main results in \cite{sawin-wood2022}. This is precisely what we do in Section \ref{sec:MBmeas} for a particular sequence of finite moments modeling the absolute Galois group, constructing the measure $\mu_K^{\rm MB}$ to prove Theorem \ref{thm:muK_construct}.

In Section \ref{sec:orderings} we translate the discriminant ordering for number fields into the language of \cite{alberts2022}, that is a sequence of $L^1$-functions $f_n$ on the underlying category of finite objects. We then determine the moments of the ordering, which agree with the sum of coefficients of the Malle-Bhargava local series. The main results of \cite{alberts2022} are then used to prove Theorem \ref{thm:prob1intro} in Section \ref{sec:proof}, as well as some suitable generalizations including the product of ramified primes ordering.

In Section \ref{sec:knownissues} we interpret the known issues of the Malle-Bhargava principle in the language of groups with local data. In some cases, we show that these issues occur with probability $0$ in $\mu_K^{\rm MB}$, suggesting that in these cases $G_K$ is ``not typical enough" to use the Vast Counting Heuristic as justification for the predicted growth rates. We do not make any conjectures in this section, but we do use this information to point towards what adjustments to the random model are likely to produce more accurate predictions.

\subsection{Notation}
\allowdisplaybreaks
\begin{align*}
G_K &= \Gal(\overline{K}/K)\text{ the absolute Galois group of }K\\
P_K &=\{\text{places of }K\}\\
P_\infty &= \{p\in P_K | p\mid \infty\}\\
G_{K_p} &= \Gal(\overline{K}_p/K_p)\text{ the absolute Galois group of }K_p\text{, where }p\text{ is a place of }K\\
I_p &=\text{the inertia group of }\overline{K}_p/K_p\text{, where }p\text{ is a place of }K\\
\Fr_p &=\text{a representative of the Frobenius element in }G_{K_p}\\
I_K &=\text{the group of fractional ideals of }K\\
|\mathfrak{a}| &=\text{the norm down to }\Q\text{ of a fractional ideal }\mathfrak{a}\in I_K\\
\disc(\pi) &=\prod_{p} p^{\ind(g_p)}\text{ where $g_p$ generates $\pi(I_p)$}\\
\ind(g) &= n - \#\{\text{orbits of }g\}\text{, where }g\in G\subseteq S_n\\
\inv &=\text{called an invariant, is some map }\prod_p \Hom(G_{K_p},G) \to I_K\\
{\rm MB}_{\inv}(K,\Sigma,s) &=\text{the Malle-Bhargava local series, see Lemma \ref{lem:originalmoment}}\\
{\rm Grp}(K) &=\text{the category of finite groups with finite $K$-local data, see Definition \ref{def:Grp}}\\
(G,S,\phi) &\phantom{=}\text{denotes an object in }{\rm Grp}(K)\\
{\rm proGrp}(K) &=\text{the category of groups with $K$-local data, see Definition \ref{def:proGrp}}\\
\mathscr{G} &\phantom{=}\text{denotes an object in }{\rm proGrp}(K)\text{ with implicit local data given by }\phi_{\mathscr{G}}\\
N(\mathscr{G},f_n) &= \sum_{(G,S,\phi)\in {\rm Grp}(K)} f_n(G,S,\phi)\#{\rm Epi}(\mathscr{G},(G,S,\phi))\text{, see \cite[Definition 1.1]{alberts2022}}\\
M &= \text{a discrete measure on ${\rm Grp}(K)$ given by a sequence of finite moments}\\
&\phantom{=}M(\{(G,S,\phi)\}) = M_{(G,S,\phi)}\\
\mu_K^{\rm MB} &=\text{the unique measure determined by Theorem \ref{thm:muK_construct}}\\
M^{(j)} &=\text{the mixed moment induced by $\mu$, see Subsection \ref{subsec:mixed}}\\
\overset{p.}{\to} &\phantom{=}\text{converges in probability}\\
\overset{a.s.}{\to} &\phantom{=}\text{converges almost surely, i.e. converges on a measure 1 set}\\
f(X) \ll g(X) &\phantom{=}\text{there exists a constant $C$ such that }f(X) \le Cg(X)\text{ for all }X\\
f(X) = O(g(X)) &\phantom{=}\text{there exists a constant $C$ such that }f(X) \le Cg(X)\text{ for all }X\\
f(X) = o(g(X)) &\phantom{=}\text{means }\frac{f(X)}{g(X)} \to 0\text{ as } X\to \infty
\end{align*}

\section*{Acknowledgments}

The author would like to thank Melanie Matchett Wood for numerous discussions on the topic and direction of this paper over the course of several years. The author also thanks Nigel Boston, Yuan Liu, Peter Koymans, and Frank Thorne for helpful conversations and feedback.

\section{The category of groups with local data}\label{sec:localdata}

In this section we define the category of groups with local data and realize this as a category of pro-objects in the language of \cite{sawin-wood2022}. By utilizing a number of tools proven in \cite{sawin-wood2022}, we prove that this category satisfies the necessary hypotheses to apply Sawin--Wood's main results. We give a family of well-behaved sequences in Proposition \ref{prop:wb} in preparation for solving the moment problem in Section \ref{sec:MBmeas}.

\subsection{The categories of groups with local data}

We make the following precise definition for groups with local data:

\begin{definition}\label{def:proGrp}
We let ${\rm proGrp}(K)$ denote the category of \textbf{profinite groups with $K$-local data}.
\begin{itemize}
\item[(a)] The objects of this category are pairs $(G,\phi)$ of a profinite group $G$ with a family $\phi = (\phi_p)$ of continuous homomorphisms $\phi_p:G_{K_p}\to G$ for each place $p$ of $K$.
\item[(b)] A morphism $\pi:(G,\phi) \to (H,\psi)$ is a continuous homomorphism $\pi:G\to H$ such that $\pi\phi_p = \psi_p$ for each place $p$ of $K$.
\end{itemize}
We will often refer to the objects as just ``groups with local data" when $K$ is clear from context, with the profiniteness being left implicit.
\end{definition}

Any Galois extension of $K$ comes with not just a Galois group, but a Galois group with local data given by $(\Gal(L/K),\phi_{L/K})$ where $\phi_{L/K}|_{G_{K_p}}$ is given by the corresponding local extension $G_{K_p} \to \Gal(L_p/K_p) \hookrightarrow \Gal(L/K)$.

\textbf{Remark:} Technically, a Galois group with local data is only well-defined up to conjugation of the image of each $\phi_{L/K,p}$. We fix throughout a choice of embedding $G_{K_p}\hookrightarrow G_K$ for each place $p$ of $K$ so that we can specify the Galois group with local data explicitly. This is mostly for convenience - the results of this paper will still hold without making this choice as long as all orderings and local conditions are chosen to be conjugation invariant. However, the work is significantly easier to follow if we do not have an extra conjugation relation floating around.

We want to apply the results of \cite{sawin-wood2022} to ${\rm proGrp}(K)$, however this category has uncountably many isomorphism classes. Thus, we consider the pro-objects case in \cite[Theorem 1.7 and 1.8]{sawin-wood2022}, which makes sense as we allowed profinite groups in ${\rm proGrp}(K)$. In order to apply these results, we need to find a category of finite objects for which ${\rm proGrp}(K)$ is the corresponding category of pro-objects.

It will not be enough to just restrict to pairs $(G,\phi)$ with $G$ finite, as this will still have uncountably many objects. We also need to restrict the places at which we have local data.

\begin{definition}\label{def:Grp}
We let ${\rm Grp}(K)$ denote the category of \textbf{finite groups with finite $K$-local data}.
\begin{enumerate}
\item[(a)] The objects of this category are pairs $(G,S,\phi)$ of a finite group $G$, a finite set of places $S$ of $K$, and a family $\phi = (\phi_p)_{p\in S}$ of continuous homomorphisms $\phi_p:G_{K,S}\to G$ for each place $p\in S$.
\item[(b)] A morphism $\pi:(G,S,\phi) \to (H,S',\psi)$ is a continuous homomorphism $\pi:G\to H$ such that
\begin{itemize}
\item $S'\subseteq S$,
\item For each place $p\in S\cap S'$, $\pi\phi_p = \psi_p$, and
\item For each place $p\in S\setminus S'$, $\pi\phi_p(I_p) = 1$.
\end{itemize}
\end{enumerate}
We will often refer to the objects as just ``finite groups with finite local data" when $K$ is clear from context.
\end{definition}

It is clear that ${\rm Grp}(K)$ has only countably many isomorphism classes, as there are countably many finite groups $G$, countably many finite sets of places $S$, and for each $G$ and $p\in S$ the set $\Hom(G_{K_p},G)$ is finite. Our definition of morphism reflects what we want out of this category: morphisms can only pass from local data at more places to less places, reflecting that in the inverse limit we want to get local data at all places. The fact that we ask $\pi\phi$ to be unramified at any place $p\not\in S'$ is a bit more subtle. There are two reasons for this:
\begin{itemize}
\item We want ramification data to be preserved so that these finite objects play nice with discriminants. In particular, we want any epimorphism $(G,S,\phi) \to (G,S',\psi)$ restricting to the identity on $G$ to not forget inertia data. This will imply that whenever such an epimorphism exists, $\disc(G,S,\phi) = \disc(G,S',\psi)$.
\item Why not require that $\pi\phi(G_{K_p}) = 1$? This would be too restrictive. In the inverse limit with this property, only groups with local data that are totally split at all but finitely many places would occur as pro-objects. By the Chebotarev density theorem, this will exclude all Galois groups with local data and so miss the very structure we are attempting to model.
\end{itemize}

We give a brief summary of the notion of level and the topology on these categories as defined in \cite{sawin-wood2022}. For the most part we will be able to directly cite results of \cite{sawin-wood2022}, but there will occaisionally be times that we need to delve into the specifics of this topology. The notion of level, in particular, is important for working with pro-objects.

A \textbf{level} of ${\rm Grp}(K)$ is a subset of the isomorphism classes of ${\rm Grp}(K)$, $\mathcal{C}$, which is the smallest subset containing some finite set of isomorphism classes which is downward-closed and join-closed, where
\begin{itemize}
\item downward closed means that if $(G,S,\phi)\in \mathcal{C}$ and ${\rm Epi}((G,S,\phi),(H,S',\varphi))\ne \emptyset$ then $(H,S',\varphi)\in \mathcal{C}$, and
\item join closed means that for any finite object $(G,S,\phi)$, if $(H_1,S_1,\phi_1)$ and $(H_2,S_2,\phi_2)$ are quotients of $(G,S,\phi)$ (i.e. there exists an epimorphism to them) with both belonging to $\mathcal{C}$, then so does the join $(H_1,S_1,\phi_1)\vee (H_2,S_2,\phi_2)$, taken as the join in the lattice of quotients of $(G,S,\phi)$.
\end{itemize}

The \textbf{level topology} on either ${\rm proGrp}(K)$ or ${\rm Grp}(K)$ is defined by taking basic opens
\[
U_{\mathcal{C}, \mathscr{H}} = \{\mathscr{G} : \mathscr{G}^{\mathcal{C}} = \mathscr{H}\},
\]
where $\mathcal{C}$ is a level, $\mathscr{H}\in\mathcal{C}$, and $\mathscr{G}^{\mathcal{C}}$ is the maximal quotient of $\mathscr{G}$ belonging to $\mathcal{C}$, or equivalently the join of every element of $\mathcal{C}$ below $\mathscr{G}$ in the lattice of quotients of $\mathscr{G}$.

We now prove that these categories satisfy the precise conditions needed for the tools in \cite{sawin-wood2022}.

\begin{proposition}\label{prop:progrpK}
Let $K$ be a number field. Then
\begin{itemize}
\item[(a)] ${\rm Grp}(K)$ is a diamond category \cite[Definition 1.3]{sawin-wood2022}, and
\item[(b)] ${\rm proGrp}(K)$ is (isomorphic to) the subcategory of pro-objects of ${\rm Grp}(K)$ for which every place $p$ of $K$ is defined in the local data of some finite quotient \cite[Section 1.2]{sawin-wood2022}.
\end{itemize}
\end{proposition}

The category of pro-objects of ${\rm Grp}(K)$ can be shown, by the same proof as below, to be isomorphic to the category of objects $(G,S,\phi)$ for $G$ a profinite group, $S$ \emph{any} set of places of $K$, and $\phi=(\phi_p)_{p\in S}$ a family continuous homomorphisms $\phi_p:G_{K_p}\to G$. The probability measure we define will be supported on the subcategory ${\rm proGrp}(K)$ so it is not necessary to consider the full category of pro-objects.

\begin{proof}
Sawin--Wood prove extremely general tools for the recognition of diamond categories. We will use three of their results here.

Sawin--Wood prove that the category of finite groups ${\rm Grp}$ is a diamond category in \cite[Lemma 6.19]{sawin-wood2022}. Given the set $P_K$ of places of $K$, let $\mathcal{P}_K$ be the opposite category of the injective category whose objects are finite subsets of $P_K$ and whose morphisms are inclusion maps. In this category, $\Hom(S,S')$ is either empty, or contains only the embedding $S\hookleftarrow S'$. This category trivially satisfies the properties of a diamond category.

The product category ${\rm Grp}\times \mathcal{P}_K$ is then a diamond category by \cite[Lemma 6.16]{sawin-wood2022}. The local data $\phi$ can be seen as some ``finite data" in this category. Let $\mathcal{G}:{\rm Grp}\times \mathcal{P}_K \to {\rm FinSet}$ be the functor sending
\[
(G,S) \mapsto \prod_{p\in S}\Hom(G_{K_p},G).
\]
Then the category $({\rm Grp}\times \mathcal{P}_K, \mathcal{G})$ of pairs $(G,S,\phi)$ of $(G,S)\in {\rm Grp}\times \mathcal{P}_K$ together with $\phi\in \mathcal{G}(G,S)$ is a diamond category by \cite[Lemma 6.21]{sawin-wood2022}. This is precisely ${\rm Grp}(K)$.

A pro-object of ${\rm Grp}(K)$ is defined in \cite[Subsection 1.2]{sawin-wood2022} to be a sequence $X = (X^{\mathcal{C}})$ indexed by levels $\mathcal{C}$ for which $(X^{\mathcal{C}'})^{\mathcal{C}} = X^{\mathcal{C}}$ whenever $\mathcal{C}\subset \mathcal{C}'$. Let $\mathcal{P}({\rm Grp}(K))$ denote the category of pro-objects of ${\rm Grp}(K)$.

There certainly exists a functor $F:{\rm proGrp}(K) \to \mathcal{P}({\rm Grp}(K))$ defined by
\[
(G,\phi) \mapsto ( (G,S,\phi)^{\mathcal{C}} )
\]
and
\[
\pi \mapsto \pi^{\mathcal{C}}.
\]
This is essentially a tuple of forgetful functors from ${\rm proGrp}(K)$ to the level $\mathcal{C}$ for each level. The image of this functor is precisely those pro-objects that involve local data at all places. Let $\mathcal{D}$ denote this category.

The inverse functor $F^{-1}:\mathcal{D} \to {\rm proGrp}(K)$ is given by the inverse limit. If we write $X^{\mathcal{C}}$ as $(G_{\mathcal{C}}, S_{\mathcal{C}}, \phi_{\mathcal{C}})$, then
\[
X \mapsto \left(\lim_{\substack{\leftarrow\\\mathcal{C}}} G_{\mathcal{C}}, \bigcup_{\mathcal{C}} S_{\mathcal{C}}, \lim_{\substack{\leftarrow\\\mathcal{C}}} \phi_{\mathcal{C}}\right)
\]
and
\[
(\pi^{\mathcal{C}}) \mapsto \lim_{\substack{\leftarrow\\\mathcal{C}}} \pi^{\mathcal{C}}
\]
It is clear that $F^{-1}\circ F$ is the identity functor. Given that the subcategory $\mathcal{D}$ of $\mathcal{P}({\rm Grp}(K))$ consists of precisely those objects for which $\bigcup S_{\mathcal{C}}$ is the set of all places, we see that $F\circ F^{-1}$ is also the identity functor.
\end{proof}

\subsection{Well-behaved sequences}

The results of Sawin--Wood \cite{sawin-wood2022} apply to sequences of moments which are ``well-behaved", i.e. they do not grow to fast. More explicitly, Sawin--Wood call a sequence of finite moments $M_{(G,S,\phi)}$ ``well-behaved" if the series
\[
\sum_{(G,S,\phi)\in \mathcal{C}} \sum_{\pi\in \Surj((G,S,\phi),(F,S',\psi))} \frac{|\mu((F,S',\psi),(G,S,\phi))|}{|\Aut(G,S,\phi)|} Z(\pi)^3 M_{(G,S,\phi)}.
\]
is absolutely convergent, where $\mu(A,B)$ is the M\"obius function on the lattice of quotients, $Z(\pi)$ is the number of elements between $(G,S,\phi)$ and $(F,S',\psi)$ which satisfy the lattice distributive law, and $M_{(G,S,\phi)}$ are the moments in question.

\begin{proposition}\label{prop:wb}
Let $M_{(G,S,\phi)}$ be a sequence of finite moments on the isomorphism classes of ${\rm Grp}(K)$. Suppose there exist real constants $f(S)$ and $e(S)$depending only on $S$ such that $M_{(G,S,\phi)} = O(f(S)|G|^{e(S)})$. Then the sequence $M_{(G,S,\phi)}$ is well-behaved in the sense of \cite{sawin-wood2022}.
\end{proposition}

Proposition \ref{prop:wb} can be seen as an analog for the corresponding result for groups: Sawin--Wood prove in \cite[Corollary 6.13]{sawin-wood2022} that if $M_G = O(|G|^n)$ for some real number $n$, then $M_G$ is well-behaved in the category of finite groups. Proposition \ref{prop:wb} is essentially the same strength as this, requiring very little control control as $S$ varies and no control as $\phi$ varies.

\begin{proof}
In practice, checking well-behavedness might be a bit of a chore. Sawin--Wood prove some useful tools for us to shorten this process. Recall that the category ${\rm Grp}(K)$ is given by $({\rm Grp}\times \mathcal{P}_K, \mathcal{G})$ for the functor of finite data $\mathcal{G}(G,S) = \prod_p\Hom(G_{K_p},G)$. The case of well-behavedness in a category with finite data is already studied by Sawin--Wood in \cite[Lemma 6.22]{sawin-wood2022}. The sequence $M_{(G,S,\phi)} = O(f(S)|G|^{e(S)})$ is well-behaved if the sequence
\begin{align*}
\sum_{\phi\in \prod_p \Hom(G_{K_p},G)} M_{(G,S,\phi)} &= O\left(f(S)\prod_{p\in S}|\Hom(G_{K_p},G)||G|^{e(S)}\right)
\end{align*}
is well-behaved in ${\rm Grp}\times \mathcal{P}_K$.

Sawin--Wood do not address well-behavedness in product categories, but many of the features of the well-behavedness sum factor over the product. Consider that in the product category we necessarily have
\begin{align*}
\Aut(G,G') &= \Aut(G) \times \Aut(G'),\\
\mu((F,F'),(G,G')) &= \mu(F,G)\mu(F',G'),\\
Z(\pi_1,\pi_2) &= Z(\pi_1)Z(\pi_2).
\end{align*}
Moreover, each level in the product category is contained in a product of levels $\mathcal{C}_1\times \mathcal{C}_2$ from the individual categories. One immediately proves the following result:

\begin{lemma}\label{lem:wb_prod}
Let $C_1$ and $C_2$ be two diamond categories. If the sequences $(M_G)_{G\in C_1}$ and $(M_{G'})_{G'\in C_2}$ are well-behaved in their respective categories, then the sequence $(M_G M_{G'})_{(G,G')\in C_1\times C_2}$ is well-behaved in the product category.
\end{lemma}


We leave the details of the proof to the interested reader, as we are not actually able to use this result. We make no such assumption that our moments sequence factors over the product, and in fact the upper bound $f(S)|G|^{e(S)}$ does not factor over the product. Luckily, it turns out that the category $\mathcal{P}_K$ is \emph{particularly} nice for the well-behavedness property.

\begin{lemma}\label{lem:wb_prod_special}
Let $C$ be a diamond category and $\mathcal{N}$ be the opposite category of finite subsets of $\mathbb{N}$ under inclusion. The sequence $M_{(G,S)}$ is well-behaved in the category $C\times \mathcal{N}$ if, for each fixed object $S\in \mathcal{N}$, the sequence $M_{(G,S)}$ is well-behaved in $C$.
\end{lemma}

Here we remark that $\mathcal{P}_K$ and $\mathcal{N}$ are isomorphic as categories, regardless of the choice of base field $K$. This isomorphism comes from a choice of bijection from the countable set $P_K$ to $\N$.

\begin{proof}
Any level $\mathcal{C}$ of $\mathcal{N}$ consists solely of the finitely many subsets of some finite set $S\subseteq \mathbb{N}$. In the category $\mathcal{N}$ every morphism is an epimorphism, and for any object $S$ the epimorphisms correspond precisely to the finitely many subsets of $S$. 
Thus, for any product level $\mathcal{C} = \mathcal{C}_1\times \mathcal{C}_2$ we separate the well-behavedness sum as
\begin{align*}
&\sum_{(G,S)\in\mathcal{C}} \sum_{\pi\in\Surj((G,S),(F,S'))} \frac{|\mu((F,S'),(G,S))|}{|\Aut(G,S)|}Z(\pi)^3M_{(G,S)}\\
&=\sum_{S\in\mathcal{C}_2}\sum_{\pi_2\in\Surj(S,S')} \frac{|\mu(S',S)|}{|\Aut(S)|}Z(\pi_2)\left(\sum_{G\in\mathcal{C}_1} \sum_{\pi_1\in\Surj(G,F)} \frac{|\mu(F,G)|}{|\Aut(G)|}Z(\pi_1)^3M_{(G,S)}\right).
\end{align*}
The first two summations are finite, and the inner two summations are absolutely convergent by the well-behavedness of $M_{(G,S)}$ in $C$ for each object $S$. Thus the entire summation is convergent.
\end{proof}

By \cite[Corollary 6.13]{sawin-wood2022}, any sequence $M_G = O(|G|^n)$ for some real number $n$ is well-behaved over ${\rm Grp}$. It is known that the decomposition groups $G_{K_p}$ have finite rank for each place $p$, depending on $K$ and $p$. Therefore
\begin{align*}
M_{(G,S)} &= \sum_{\phi\in \prod_p \Hom(G_{K_p},G)} M_{(G,S,\phi)}\\
&=O\left(f(S)\prod_p |\Hom(G_{K_p},G)||G|^{e(S)}\right)\\
&=O_S\left(|G|^{e(S)+O_{K,S}(1)}\right),
\end{align*}
which is necessarily well-behaved in ${\rm Grp}$ for each fixed $S\in \mathcal{P}_K$ (with $K$ fixed throughout). Thus by Lemma \ref{lem:wb_prod_special}, it is well-behaved in ${\rm Grp}\times \mathcal{P}_K$, concluding the proof of well-behavedness of the sequence $M_{(G,S,\phi)}$.
\end{proof}

\section{Constructing the Malle-Bhargava measure}\label{sec:MBmeas}

In this section, we prove Theorem \ref{thm:muK_construct} giving the existence and uniqueness of a probability measure $\mu_K^{\rm MB}$ modeling the absolute Galois group. We do this by constructing a sequence of finite moments $M_G$ out of the Chebotarev density theorem, then applying \cite[Theorem 1.8]{sawin-wood2022}. This result states that given any sequence of measures on the isomorphism classes of pro-objects whose finite moments converges to a well-behaved sequence, the measures themselves weakly converge to a unique measure on the isomorphism classes of pro-objects.

\subsection{A sequence of measures approximating the absolute Galois group}

Write $P_{\infty} = \{p\mid \infty\}$ for the set of infinite places. For $S$ a set of finite places containing all the infinite places, we define the pro-free product
\[
F_{K,S} = \bigast_{p\in S} G_{K_p}.
\]
This is \emph{not} a group with local data, as we do not have any local data at places outside of $S$. Informed by the Chebotarev density theorem and class field theory, for any tuple $(r_x)$ of elements $r_x\in F_{K,S}$ for $x\in\{0,...,|S|-1\}\cup (P_K \setminus S)$ we define the quotient
\[
\mathscr{F}_{K,S}(r) = F_{K,S} / \langle r_0,...,r_{|S|-1}\rangle.
\]
The $r_1,...,r_{|S|-1}$ come from class field theory, given by the rank of the $S$-unit group $\mathcal{O}_K^{\times}$. These $|S| - 1 = |S\setminus P_{\infty}| + (|P_{\infty}| - 1)$ relations correspond to the $n+u$ relations from \cite{liu-wood2020,liu-wood-zureick-brown2019} for $n$ corresponding to $|S\setminus P_\infty|$, the number of finite places in $S$, and $u$ corresponds to the unit rank $u_K = \rk \mathcal{O}_K^{\times} = |P_\infty| - 1$.

The relation $r_0$ models the relations in class field theory coming from the roots of unity in the base field $\mu(K)$. This relation does not appear in \cite{liu-wood-zureick-brown2019} as they only consider unramified extensions whose order is coprime to $|\mu(K)|$.

The elements $r_p$ of the tuple for $p\not\in S$ do not contribute relations to the underlying group, but instead are used to specify local data. This gives $\mathscr{F}_{K,S}(r)$ canonical local data $\phi = (\phi_p)$ by
\begin{itemize}
\item $\phi_p$ is the composition $G_{K_p} \hookrightarrow F_{K,S} \to \mathscr{F}_{K,S}(r)$ if $p\in S$, and
\item $\phi_p|_{I_p}=1$ and $\phi_p(\Fr_p) = r_p$ if $p\not\in S$.
\end{itemize}
Informed by Cheboterav density (which states that Frobenius elements vary Haar randomly within the absolute Galois group) and class group heuristics (which predicts that the unit group embeds Haar randomly into the ideles) we define the following probability measures.

\begin{definition}\label{def:muKS}
Let $S$ be a finite set of places containing all infinite place. Define
\[
\mu_{K,S}^{\rm MB}(A) = {\rm Prob}\left(\mathscr{F}_{K,S}(r) \in A\right)
\]
for any set $A$ in the Borel $\sigma$-algebra of ${\rm proGrp}(K)$, where each of the $r_x$ are taken to vary independently Haar random from the following spaces:
\begin{itemize}
\item[(a)] $r_0$ is taken to vary Haar randomly in the preimage of $F_{K,S}^{\rm ab}[|\mu(K)|]$ in $F_{K,S}$ under the abelianization map, and
\item[(b)] $r_1,r_2,...,r_{|S|-1}$ and $r_p$ for $p\not\in S$ are taken to be independently Haar random in $F_{K,S}$.
\end{itemize}
\end{definition}

Let ${\rm ab}:F_{K,S}\to F_{K,S}^{\rm ab}$. The distinct behavior of $r_0$ is to ensure that there is a surjective homomorphism $\mathcal{O}_{K,S}^{\times} \to {\rm ab}(\langle r_0,r_1,...,r_{|S|-1}\rangle)$ from the $S$-units of $K$, where a generator of $\mu(K)$ is sent to $r_0$ and a basis for the free part of $\mathcal{O}_{K,S}^{\times}$ is sent to $r_1,...,r_{|S|-1}$. The ranks agree by Dirichlet's unit theorem, noting that $P_\infty\subset S$. By local class field theory, the abelianization is given by
\[
F_{K,S}^{\rm ab} \cong \prod_{p\in S} K_p^{\times},
\]
Varying the relations $r_0,...,r_{|S|-1}$ Haar randomly corresponds to choosing a Haar random homomorphism $\mathcal{O}_{K,S}^{\times} \to \prod_{p\in S} K_p^{\times}$, thus choosing a random image of the $S$-units.

This construction is built with Malle's conjecture in mind. For a given real number $X$, there are only finitely many extensions $L/K$ with discriminant bouned above by $X$. Let $L_X$ be the compositum of all such extensions. Then the Galois group with local data $(\Gal(L_X/K),\phi_{L_X/K})$ is a quotient of $\mathscr{F}_{K,S}(r)$ for at least one nontrivial tuple $r$, where $S=S_X$ is chosen large enough to generate $\Gal(L_X\cap K^{ur}/K)$ and $r$ is some tuple defining relations compatible with $G_K$.

The philosophy from class group statistics that the unit group has Haar random image in the group of $S$-ideals informs the choice to vary $r_0,r_1,...,r_{|S|-1}$ randomly in the model. The Chebotarev density theorem informs the choice to allow $r_p$ for $p\not\in S$ to vary Haar randomly in the model. Put together, this heuristic reasoning aligns with $\mu_{K,S}^{\rm MB}$. In the limit as $X\to \infty$, we will need larger and larger sets of place $S_X$, so it makes sense to model $G_K$ with a limit of $\mu_{K,S}^{\rm MB}$ as $S$ tends towards the set of all places, $P_K$.

\subsection{The proof of Theorem \ref{thm:muK_construct}}

Constructing a measure $\mu_K^{\rm MB}$ that is the limit of $\mu_{K,S}^{\rm MB}$ is precisely the purpose of \cite[Theorem 1.8]{sawin-wood2022}. It will suffice to compute the finite moments of $\mu_{K,S}^{\rm MB}$ in the limit as $S$ tends towards the set of all places, which we will check is well-behaved using Proposition \ref{prop:wb} and corresponds to a unique measure using \cite[Theorem 1.8]{sawin-wood2022}.

\begin{proposition}\label{prop:muKSmoments}
Let $K$ be a number field and $(G,S,\phi)\in {\rm Grp}(K)$. Then for any set of places $S'\supseteq S$
\[
\int \#{\rm Epi}(\mathscr{G},(G,S,\phi))\ d\mu_{K,S'}^{\rm MB} = |G^{\rm ab}[|\mu(K)|]|^{-1}|G|^{-|S\cup P_\infty|+1}.
\]
\end{proposition}

This proposition is the source of the finite moments in Theorem \ref{thm:muK_construct}. We remark that this result is stronger than evaluating the limit of finite moments. $S'\supseteq S$ will eventually be true in the limit so that
\[
\lim_{S'\to P_K} \int \#{\rm Epi}(\mathscr{G},(G,S,\phi))\ d\mu_{K,S'}^{\rm MB} = |G^{\rm ab}[|\mu(K)|]|^{-1}|G|^{-|S\cup P_\infty|+1}
\]
because the sequence is eventually constant.

\begin{proof}
For a fixed $(G,S,\phi)\in {\rm Grp}(K)$, we consider that if $S\subseteq S'$ then
\begin{align*}
\int \#{\rm Epi}(\mathscr{G},(G,S,\phi))\ d\mu_{K,S'}^{\rm MB}&= \sum_{\substack{\varphi\in \Hom(F_{K,S'},G)\\\varphi|_{F_{K,S}} = \phi\\\varphi(I_p) = 1\text{ if } p\in S'\setminus S}} \int\#{\rm Epi}(\mathscr{G},(G,S',\varphi))\ d\mu_{K,S'}^{\rm MB}.
\end{align*}
Each $\phi$ can be understood as a homomorphism $F_{K,S} \to G$. The set ${\rm Epi}(\mathscr{F}_{K,S'}(r),(G,S',\varphi))$ can have at most one element, given by $\varphi$ if $\varphi$ factors through the quotient $F_{K,S'} \to \mathscr{F}_{K,S'}(r)$. This happens if and only if each relation belongs to the kernel. The relations vary independently Haar randomly, so it follows that
\begin{align*}
\int\#{\rm Epi}(\mathscr{G},((G,S',\varphi),\pi))\ d\mu_{K,S'}^{\rm MB}&= \prod_{i=0}^{|S'|-1} \mu_{Haar}\left(\ker q_*\varphi \right)\\
&= |G^{\rm ab}[|\mu(K)|]|^{-1}|G|^{-|S'|+1}.
\end{align*}
There are precisely $|G|$ unramified continuous homomorphisms $G_{K_p} \to G$ for finite places, so the summation includes an extra extra factor of $|G|$ for each $p\in S' \setminus (S\cup P_\infty)$. Thus the integral is given by
\begin{align*}
\int \#{\rm Epi}(\mathscr{G},(G,S,\phi))\ d\mu_{K,S'}^{\rm MB} &= |G^{\rm ab}[|\mu(K)|]|^{-1}|G|^{|S'\setminus (S\cup P_\infty)| - |S'| +1}\\
&= |G^{\rm ab}[|\mu(K)|]|^{-1}|G|^{-|S\cup P_\infty|+1}.
\end{align*}
\end{proof}

We are now ready to prove Theorem \ref{thm:muK_construct}.

\begin{proof}[Proof of Theorem \ref{thm:muK_construct}]
The sequence of moments
\[
M_{(G,S,\phi)} = |G^{\rm ab}[|\mu(K)|]|^{-1}|G|^{-|S\cup P_\infty|+1} = O(|G|^0)
\]
satisfies the hypothesis of Proposition \ref{prop:wb}, and so is well-behaved in the sense of \cite{sawin-wood2022}. Thus, Proposition \ref{prop:muKSmoments} and \cite[Theorem 1.8]{sawin-wood2022} together imply the existence of a measure $\mu_K^{\rm MB}$ on the category of pro-objects of ${\rm Grp}(K)$ for which $\mu_{K,S}^{\rm MB}\to \mu_K^{\rm MB}$ weakly converges as $S\to P_K$ and $\mu_K^{\rm MB}$ has the prescribed finite moments. As the $\mu_{K,S}^{\rm MB}$ have total measure $1$, so too does $\mu_K^{\rm MB}$ making it a probability measure. Thus, it suffices to show that $\mu_{K}$ is supported on ${\rm proGrp}(K)$.

It is the case that $\mu_{K,S}^{\rm MB}$ are supported on ${\rm proGrp}(K)$ by construction, so it is tempting to say that $\mu_K^{\rm MB}$ is as well because of weak convergence. However, this is not a property held by weakly convergent sequences of measures in general. The fact that $P_K$ is countable saves us here. For each place $p$ of $K$, let $f_p$ be the function on the pro-objects of ${\rm Grp}(K)$ defined by
\[
f_p(G,S,\phi) = \begin{cases}
1 & p\not\in S\\
0 & p\in S,
\end{cases}
\]
where we recall that a pro-object can have local data at any set of places $S$. This function is continuous, as
\[
f_p^{-1}(0) = \bigcup_{\mathcal{C}} \bigcup_{\substack{(G,S,\phi)\in\mathcal{C}\\p\in S}} U_{\mathcal{C},(G,S,\phi)}
\]
is a union of basic opens and
\[
f_p^{-1}(1) = \bigcup_{\mathcal{C},\ p\in S(\mathcal{C})} \bigcup_{\substack{(G,S,\phi)\in\mathcal{C}\\p\not\in S}} \left(U_{\mathcal{C},(G,S,\phi)} \setminus \bigcup_{\substack{(G,S\cup\{p\},\psi)\in \mathcal{C}\\\psi|_{S} = \phi}} U_{\mathcal{C},(G,S\cup\{p\},\psi)}\right),
\]
where $S(\mathcal{C})$ is the set of primes appearing in at least one isomorphism class in the level $\mathcal{C}$. The inner most union is a finite union of basic opens, which are in fact clopen in the level topology defined by \cite{sawin-wood2022}. Thus, the set difference is open, and this preimage is again the union of open sets. Thus $f_p$ lifts to a continuous function on the pro-objects. The expected value of bounded continuous functions converges along weakly convergent sequences of measures. The fact that $\mu_{K,S}^{\rm MB}$ is supported on ${\rm proGrp}(K)$ implies
\[
\int f_p\ d\mu_K^{\rm MB} = \lim_{S\to P_K}\int f_p\ d\mu_{K,S}^{\rm MB} = 0.
\]
Thus, the set of pro-objects without local data at $p$ is a null set. There are only countably many places $p$, so by countable additivity we find that the complement of ${\rm proGrp}(K)$ is a null set, i.e. $\mu_K^{\rm MB}$ is supported on ${\rm proGrp}(K)$.
\end{proof}

\section{Outlining the Proof of Theorem \ref{thm:prob1intro}}

Malle's classical counting function for the transitive subgroup $G\subseteq S_n$ is achieved by the ordering $\disc_X^G:{\rm Grp}(K)\to \R$ defined by
\[
\disc_X^G(H,S,\phi) = \begin{cases}
1 & H\cong G,\ |\disc(\phi)| \le X \text{, and } S=\{|p|\le X\}\\
0 & \text{else}.
\end{cases}
\]
The author defines the corresponding counting function on a category in \cite{alberts2022} by
\[
N(\mathscr{G},\disc_X^G) := \sum_{(G,S,\phi)\in{\rm Grp}(K)} \disc_X^G(G,S,\phi) \#{\rm Epi}(\mathscr{G},(G,S,\phi)).
\]
Theorem \ref{thm:prob1intro} will be proven utilizing the results of \cite{alberts2022}. This will take place in (roughly) three important steps:
\begin{enumerate}[(I)]
\item\label{it:counting_function} We first prove that Malle's counting function agrees with that of $N(\mathscr{G},\disc_X^G)$, that is
\[
N(\mathscr{G},\disc_X^G) = \#\{\pi\in \Surj(\mathscr{G},G) : |\disc(\pi)|\le X\}.
\]
This will be a consequence of Lemma \ref{lem:counting_function}, proven in Section \ref{sec:orderings}.

\item\label{it:MB_local_series} \cite[Theorem 1.3]{alberts2022} states that, under suitable conditions, the counting function $N(\mathscr{G},\disc_X^G)$ is asymptotic to
\[
\int_{{\rm Grp}(K)}\disc_X^G\ dM := \sum_{(G,S,\phi)\in{\rm Grp}(K)} \disc_X^G(G,S,\phi) M_G
\]
as $X$ tends to infinity. This intergal is computed in Lemma \ref{lem:originalmoment}, showing that it is given by the sum of coefficients of the Malle-Bhargava local series. This is the source of Malle's predicted asymptotic growth rate.

\textbf{Remark:} Technically, \cite{alberts2022} requires $X$ to be a positive integer and not a real number. This is OK for our purposes as the norm of the discriminant is always positive integer valued. We will abuse notation throughout by taking $X$ to be a positive integer in order to align with the standard notation of Malle's conjecture.

\item\label{it:mixed_moments} In order to verify the conditions of \cite[Theorem 1.3]{alberts2022}, we appeal to \cite[Theorem 1.4]{alberts2022} and \cite[Corollary 1.5]{alberts2022}. It will suffice to bound integrals of the form
\[
\int_A \disc_X^G\ dM^{(j)}(dM)^{2k-j} := \sum_{(G_i,S_i,\phi_i)\in A} \prod_{i=1}^{2k}\left(\disc_X^G(G_i,S_i,\phi_i)\right) M^{(j)}_{(G_i,S_i,\phi_i)_{i=1}^j} \prod_{i=j+1}^{2k} M_{(G_i,S_i,\phi_i)}
\]
for particular subcategories $A\subseteq {\rm Grp}(K)^{2k}$. Here, the mixed moments $M^{(j)}:{\rm Grp}(K)^j/\cong \to [0,\infty]$ induced by $\mu_K^{\rm MB}$ are defined by
\[
M^{(j)}_{((G_1,S_1,\phi_1),...,(G_j,S_j,\phi_j))} := \int_{{\rm proGrp}(K)} \left(\prod_{i=1}^j \#{\rm Epi}(\mathscr{G},(G_i,S_i,\phi_i))\right)\ d\mu_K^{\rm MB}(\mathscr{G}).
\]
A priori these moments need not be finite, although we will prove that they are in the cases we are concerned with. This will be the focus of Section \ref{sec:proof}, where the mixed moments will be bounded in Lemma \ref{lem:mixedmoment}. Theorem \ref{thm:prob1intro} will then follow.
\end{enumerate}

Lemma \ref{lem:mixedmoment}, used for the final step, is the most technical result of this paper. As it turns out, the mixed moments of $\disc_X^G$ are naturally related to counting surjections ordered by other, non-discriminant invariants with restricted ramification behavior. For this reason, it is natural (and, in fact, required) to work in a more general setting than that of Theorem \ref{thm:prob1intro}.

For this reason, we structure Section \ref{sec:orderings} in terms of admissible invariants and arbitrary restricted local conditions of which the discriminant ordering will be a special case. In particular, Lemma \ref{lem:counting_function} is stated in full generality for admissible invariants and admissible local conditions. This completely generalizes steps (\ref{it:counting_function}) and (\ref{it:MB_local_series}).

We prove Theorem \ref{thm:prob1intro} in Section \ref{sec:proof}, which is almost entirely focused on proving Lemma \ref{lem:mixedmoment} to complete step (\ref{it:mixed_moments}). These methods are inherently general, and the reader is right to suspect that similar results are true at the level of generality we present in Lemma \ref{lem:originalmoment}. However, in more general circumstances the analysis becomes technically complicated and sometimes require extra assumptions on the ordering and/or local conditions. Despite being the most technical part of the paper, Lemma \ref{lem:mixedmoment} is where the magic happens and is the reason we can apply the results of \cite{alberts2022}. In an effort to not obscure the ideas behind Lemma \ref{lem:mixedmoment}, we restrict to the simplest cases available.

\section{Multiplicative orderings and the Malle-Bhargava local series}\label{sec:orderings}

The goal of this section is to complete steps (\ref{it:counting_function}) and (\ref{it:MB_local_series}) in full generality. We will define the notions of an admissible invariant and an admissible family of local conditions to work over, and prove Lemma \ref{lem:originalmoment} giving the explicit correspondence between orderings and the Malle-Bhargava local series.

\subsection{Ordering by admissible invariants}
We will study more general orderings as well as general local restrictions. This expanded setting has been widely considered in the study of Malle's conjecture.

\begin{definition}
Fix a finite group $G$, a finite set of places $S$, an invariant $\inv:\prod_p \Hom(G_{K_p},G) \to I_K$, and $\Sigma = (\Sigma_p)$ a family of local conditions $\Sigma_p\subseteq \Hom(G_{K_p},G)$. Then we define the corresponding ordering $\inv_{X}^{G,S,\Sigma}:{\rm Grp}(K)/\cong\to \{0,1\}$ to be the sequence of characteristic functions for the set of $(H,S',\phi)\in{\rm Grp}(K)$ for which
\begin{enumerate}
\item[(a)] $H \cong G$,
\item[(b)] $S' = S\cup\{|p|\le X\}$,
\item[(c)] $\phi\in \prod_{p\in S'} \Sigma_p$, and
\item[(d)] $|\inv(\phi)| \le X$.
\end{enumerate}
We may omit $\Sigma$ from the notation if $\Sigma_p = \Hom(G_{K_p},G)$ is trivial for all places $p$. We may omit $S$ from the notation if $S = \emptyset$.
\end{definition}

The discriminant ordering $\disc_X^G$ is a special case, with $S = \emptyset$ and $\Sigma$ the trivial family. We defined $\inv_X^{G,S,\Sigma}$ in the greatest generality possible, although not every choice of $G$, $S$, $\inv$, and $\Sigma$ will correspond to a number field counting function. For example, even when considering the counting function $\#\{\pi\in \Surj(G_K,G): |\inv(\pi)|\le X\}$ we need a Northcott property for $\inv$ to guarantee that this is finite.

We call $\inv:\prod \Hom(G_{K_p},G) \to I_K$ an \textbf{admissible invariant} if it satisfies condtions as described in \cite{alberts2020}, i.e. it satisfies
\begin{itemize}
\item{$\inv = \prod_p \inv_p$ is multiplicative,}
\item{$\inv_p(\pi)$ is determined by $\pi|_{I_p}$ and equals $1$ if and only if $\pi|_{I_p}=1$, and}
\end{itemize}

\begin{lemma}\label{lem:counting_function}
Let $G$ be a finite group, $S$ a finite set of places, $\Sigma$ a family of local conditions, and $\inv$ an \textbf{admissible} invariant. Then
\[
N(\mathscr{G},\inv_X^{G,S,\Sigma}) = \#\left\{\pi\in \Surj(\mathscr{G},G) : (\pi|_{G_{K_p}})\in \prod\Sigma_p\text{ and }|\inv(\pi)|\le X\right\}.
\]
In particular, the counting function is independent of $S$.
\end{lemma}

That the counting function is independent of $S$ will be extremely important for our proofs. The role that $S$ plays is essentially purely bookkeeping - it is useful to impose local conditions at extra places for formulating the categorical results. As we see here, and will see again in Lemma \ref{lem:originalmoment}, the choice of extra places does not affect the counting functions whatsoever.

\begin{proof}
For a fixed $S$ and $X$, consider a surjection $\pi:\mathscr{G}\to G$. Certainly this specializes to an epimorphism in ${\rm Epi}(\mathscr{G},(G,S\cup\{|p|\le X\},\pi\phi_{\mathscr{G}}))$, so this implies
\[
\Surj(\mathscr{G},G) = \bigcup_{\phi}{\rm Epi}(\mathscr{G},(G,S\cup\{|p|\le X\},\phi)).
\]
Moreover, suppose there are two epimorphisms $\pi_i\in{\rm Epi}(\mathscr{G},(G,S\cup\{|p|\le X\}, \phi_i))$ for $i=1,2$ for which their underlying group homomorphisms are equivalent to $\pi$. Then
\[
\pi_1\phi_{\mathscr{G}} = \pi\phi_{\mathscr{G}} = \pi_2\phi_{\mathscr{G}},
\]
which implies $\phi_{1,p}=\phi_{2,p}$ at every place $p\in S\cup\{|p|\le X\}$. In particular, the identity map on the level of groups induces an isomorphism $(G,S\cup\{|p|\le X\},\phi_1)\cong(G,S\cup\{|p|\le X\},\phi_2)$. Thus, the union is in fact a disjoint union
\[
\Surj(\mathscr{G},G) = \coprod_{\phi}{\rm Epi}(\mathscr{G},(G,S\cup\{|p|\le X\},\phi)).
\]
The result then follows by noting that admissibility implies $|\inv(\pi)| = |\inv(\pi\phi_{\mathscr{G}})|$, so the sum over elements with invariant bounded above by $X$ is the same on both sides.
\end{proof}

The fact that $\inv_X^{G,S,\Sigma}(A),\inv_X^{G,S,\Sigma}(B)\ne 0$ implies $\#{\rm Epi}(A,B) = 0$ is important for proving the counting function is independent of the choice of $S$, and is a rephrasing of part of the proof above. Our goal is to study these counting functions using the results of \cite{alberts2022}, which allows us to prove probability $1$ statements by studying the moments of $\inv_X^{G,S,\Sigma}$ with respect to the discrete measure $M$ induced by the finite moments of $\mu_K^{\rm MB}$. The fact that the counting function is independent of $S$ is what allows us to give $\inv$ and $\Sigma$ any behavior at all at finitely many places.

When proving asymptotic results, one often asks that $\inv$ be \textbf{Frobenian} as in \cite{alberts2021}, which guarantees that a Tauberian theorem on the Malle-Bhargava local series will produce a nice asymptotic growth rate. This includes all the discriminant orderings, and notably includes the product of ramified primes ordering. We will be more stringent than this in Section \ref{sec:proof}.

\subsection{Types of restricted local conditions}

The type of family of local conditions $\Sigma=(\Sigma_p)$ taken can have an effect on the rate of growth. In particular, one could choose $|\Sigma_p|=1$ for every place $p$. There are uncountably many ways of making such a choice, but only countably many $G$-extensions of a number field $K$, so we expect \emph{no} $G$-extensions to satisfy such a strong choice of local restrictions.

This behavior can be boiled down to how may splitting types we are allowed to control with $\Sigma$. We make the following definition to capture this distinction:

\begin{definition}
We call $\Sigma = (\Sigma_p)$ an \textbf{admissible} family of local conditions $\Sigma_p\subset \Hom(G_{K_p},G)$ if $\Hom_{ur}(G_{K_p},G)\subseteq \Sigma_p$ for all but finitely many places.
\end{definition}

Lemma \ref{lem:originalmoment} will be proven at this level of generality. As in the case for the invariants, one also often asks for $\Sigma$ to be \textbf{Frobenian} in the sense defined in \cite{alberts2021}. We will be more stringent than this in Section \ref{sec:proof}

\subsection{The first moment of $\inv_X^{G,S,\Sigma}$}

We prove that the $L^1$-norm of $\inv_{X}^{G,S,\Sigma}$ agrees with the coefficients of the Malle-Bhargava local series. The proof is the same if we only assume admissibility, so we state the result with this level of generality:

\begin{lemma}\label{lem:originalmoment}
Let $M_{(G,S,\phi)} = |G^{\rm ab}[|\mu(K)|]|\cdot|G|^{-|S\cup P_\infty|+1}$ be the finite moments of $\mu_K^{\rm MB}$. Fix a finite group $G$, a finite set of places $S$, an admissible invariant $\inv$, and a family of local conditions $\Sigma$. We let $(a_n)$ denote the Dirichlet coefficients of the Malle-Bhargava local series
\[
{\rm MB}_{\inv}(K,\Sigma,s) = \prod_p\left(\frac{1}{|G|}\sum_{f_p\in \Sigma_p} |\inv(f_p)|^{-s}\right) = \sum_{n=1}^{\infty} a_n n^{-s}.
\]
Then
\begin{enumerate}
\item[(a)] If $\Sigma$ is admissible then
\[
\int_{{\rm Grp}(K)} \inv_X^{G,S,\Sigma} dM = \frac{|G|}{|G^{\rm ab}[|\mu(K)|]|}\sum_{n\le X} a_n,
\]
\item[(b)] If $\Sigma$ is not admissible, then $a_n=0$ for all $n$ and there exists a constant $r>1$ depending only on $K$ for which
\[
\int_{{\rm Grp}(K)} \inv_X^{G,S,\Sigma} dM = O(r^{-X/\log X}).
\]
\end{enumerate}
\end{lemma}

One important consequence of Lemma \ref{lem:originalmoment} is that $\inv_X^{G,S,\Sigma}$ is an $L^1$-ordering in the sense of \cite[Definition 1.2]{alberts2022}. Thus, \cite[Lemma 3.1]{alberts2022} states that $N(\mathscr{G},\inv_X^{G,S,\Sigma})$ is well-defined and finite almost surely with respect to $\mu_K^{\rm MB}$.

\begin{proof}
Take $P_{\infty}\subseteq \{|p|\le X\}$ by convention and set $S(X) = S\cup\{|p|\le X\}$ for simplicity. The moment is equal to a finite sum
\begin{align*}
\int_{{\rm Grp}(K)} \inv_X^{G,S,\Sigma}\ dM &= \sum_{\substack{(G,S(X),\phi)\\|\inv(\phi)|\le X\\\phi\in\prod \Sigma_p}} |G^{\rm ab}[|\mu(K)|]|^{-1}|G|^{-|S(X)|+1}\\
&=\frac{|G|}{|G^{\rm ab}[|\mu(K)|]|}\sum_{\substack{(G,S(X),\phi)\\|\inv(\phi)|\le X\\\phi\in\prod \Sigma_p}} \prod_{p\in S\cup\{|p|\le X\}} |G|^{-1}
\end{align*}
For any map $f:I_p \to G$, we set $\Sigma_p(f) = \{\psi\in \Sigma_p : \psi|_{I_p} = f\}$. Then it follows that
\[
\int_{{\rm Grp}(K)} \inv_X^{G,S,\Sigma}\ dM = \frac{|G|}{|G^{\rm ab}[|\mu(K)|]|}\sum_{\substack{\phi:\bigast_p I_p\to G\\|\inv(\phi)|\le X\\\phi\in\res_{\bigast_p I_p}(\prod \Sigma_p)}} \prod_{p\in S(X)}\frac{|\Sigma_p(\phi_p|_{I_p})|}{|G|},
\]
where $\bigast_p I_p$ is the pro-free product of inertia groups.

Suppose first that $\Sigma$ is admissible to prove part (a). Take $S$ to be large enough to contain all places for which $\Sigma_p \not\supseteq \Hom_{ur}(G_{K_p},G)$, which is finite by admissibility. Given that $|\Sigma_p(1)| = |G|$ for any $p\not\in S$, we can write
\[
\int_{{\rm Grp}(K)} \inv_X^{G,S,\Sigma}\ dM = \frac{|G|}{|G^{\rm ab}[|\mu(K)|]|}\sum_{\substack{\phi:\bigast_p I_p\to G\\|\inv(\phi)|\le X\\\phi\in\res_{\bigast_p I_p}(\prod \Sigma_p)}}\prod_{p}\frac{|\Sigma_p(\phi_p|_{I_p})|}{|G|}.
\]
We consider the corresponding Dirichlet series
\[
\frac{|G|}{|G^{\rm ab}[|\mu(K)|]|}\sum_{\substack{\phi:\bigast_p I_p\to G\\\phi\in\res_{\bigast_p I_p}(\prod \Sigma_p)}} \prod_{p}\frac{|\Sigma_p(\phi_p|_{I_p})|}{|G|}\inv(\phi)^{-s}.
\]
This is a sum of multiplicative functions, and so factors as
\begin{align*}
&\frac{|G|}{|G^{\rm ab}[|\mu(K)|]|}\prod_p\left(\frac{|\Sigma_p(1)|}{|G|} + \frac{1}{|G|}\sum_{\substack{f_p\in\Hom(I_p,G)\\f_p(I_p)\ne 1}}|\Sigma_p(f)||\inv_p(f_p)|^{-s}\right)\\
&=\frac{|G|}{|G^{\rm ab}[|\mu(K)|]|}{\rm MB}_{\inv}(K,\Sigma,s),
\end{align*}
noting that $|\Sigma_p(1)|/|G| = 1$ for all but finitely many $p$ to ensure convergence of the product in a right half plane. Matching the coefficients concludes the proof of part (a).

Now, suppose $\Sigma$ is not admissible. That the Malle-Bhargava local series diverges to $0$ is clear from the fact that infinitely many constant terms are smaller than $1$ in this case. Let $A$ be the infinite set of places for which $\Sigma_p\not\supseteq \Hom_{ur}(G_{K_p},G)$. Then
\begin{align*}
\int_{{\rm Grp}(K)} \inv_X^{G,S,\Sigma}\ dM &= \frac{|G|}{|G^{\rm ab}[|\mu(K)|]|}\sum_{\substack{\phi:\bigast_p I_p\to G\\|\inv(\phi)|\le X}} \prod_{\substack{p\in A\\\phi|_{I_p}=1\\|p|\le X}} \frac{|G|-1}{|G|}\prod_{\substack{p\\p\not\in A\text{ or }\phi|_{I_p}\ne1}}\frac{|\Sigma_p(\phi|_{I_p})|}{|G|}.
\end{align*}
At most one place with $|p|\in [X/2, X]$ can be ramified in $\phi$ at a time, so by appealing to a lower bound for the prime number theorem over $K$ there exists a positive constant $a$ depending only on $K$ for which
\begin{align*}
\int \inv_X^{G,S,\Sigma}\ dM &\le  \frac{|G|}{|G^{\rm ab}[|\mu(K)|]|}\left(\frac{|G|-1}{|G|}\right)^{\frac{aX}{\log X}}\sum_{\substack{\phi:\bigast_p I_p\to G\\|\inv(\phi)|\le X}} \prod_{\substack{p\\p\not\in A\text{ or }\phi|_{I_p}\ne1}}\frac{|\Sigma_p(\phi|_{I_p})|}{|G|}.
\end{align*}
The remaining sum is multiplicative with generating Dirichlet series
\[
\prod_p\left(1 + \frac{1}{|G|}\sum_{\substack{f_p\in\Hom(I_p,G)\\f_p(I_p)\ne 1}}|\Sigma_p(f)||\inv_p(f_p)|^{-s}\right).
\]
A Tauberian theorem bounds the size of this sum above by $O(X^{1+\epsilon})$ (see, for instance, \cite[Corollary 2.4]{alberts2021}). Thus
\begin{align*}
\int_{{\rm Grp}(K)} \inv_X^{G,S,\Sigma}\ dM &\ll  X^{1+\epsilon}\left(\frac{|G|-1}{|G|}\right)^{\frac{aX}{\log X}}\\
&\ll r^{-\frac{X}{\log X}}
\end{align*}
for some constant $r > 1$.
\end{proof}

\subsection{Tauberian theorem}

Lemma \ref{lem:originalmoment} is extremely broad. Given the additional assumption that $\inv$ and $\Sigma$ are Frobenian, a Tauberian theorem can be applied to Lemma \ref{lem:originalmoment} to give the asymptotic growth rate of $\int \inv_X^{G,S,\Sigma} dM$ as $X\to \infty$. This process is done in general in \cite[Section 2]{alberts2021}.

In order to keep the proof of Lemma \ref{lem:mixedmoment} as accessible as possible, we restrict to the following case:
\begin{corollary}\label{cor:MB_Taub}
Suppose $\inv:\prod_p \Hom(G_{K_p},G) \to I_K$ is an invariant for which there exists a weight function $w:G\to \Z_{\ge 0}$ such that
\begin{itemize}
\item{$w(1) = 0$,}
\item{$w$ is constant on $K$-conjugacy classes of $G$, and}
\item{for all but finitely many places $p$, $\inv_p(\pi) = w(g)$ if and only if $\pi(I_p) = \langle g\rangle$.}
\end{itemize}
and $\Sigma = (\Sigma_p)$ a family of local conditions for which $\Sigma_p = \Hom(G_{K_p},G)$ at all but finitely many places of $K$. If $(a_n)$ are the Dirichlet coefficients of the Malle-Bhargava local series ${\rm MB}_{\inv}(K,\Sigma,s)$ then
\[
\frac{|G|}{|G^{\rm ab}[|\mu(K)|]|}\sum_{n\le X} a_n \sim c_{\inv}(K,\Sigma) X^{1/a_{\inv}(G)}(\log X) ^{b_{\inv}(K,G) - 1},
\]
where
\begin{enumerate}[(a)]
\item The minimal weight of elements in $G$ is given by
\[
a_{\inv}(G) = \min_{g\in G\setminus \{1\}} w(g),
\]
generalizing $a(G)$ in Malle's Conjecture \ref{conj:malle}.
\item The average number of ways to be ramified with minimal weight in $G$ is given by
\[
b_{\inv}(K,G) = \#\{K\text{-conjugacy classes }\kappa\subseteq G\text{ of minimal weight }w(\kappa) = a_{\inv}(G)\},
\]
generalizing $b(K,G)$ in Malle's Conjecture \ref{conj:malle}.
\item The leading coefficient is given by a convergent Euler product with factors accounting for units
\begin{align*}
c_{\inv}(K,\Sigma) = &\frac{\left(\underset{s=1}{\rm Res}\ \zeta_K(s)\right)^{b_{\inv}(K,G)-1}}{a_{\inv}(G)^{b_{\inv}(K,G)-1}(b_{\inv}(K,G)-1)!|G^{\rm ab}[|\mu(K)|]|\cdot|G|^{u_K}}\prod_{p\mid \infty} |\Sigma_p|\\
&\cdot \prod_{p\nmid \infty} \left[(1-p^{-1})^{b_{\inv}(K,C)}\left(\frac{1}{|G|}\sum_{f\in \Sigma_p} p^{-\frac{\nu_p\inv(f)}{a_{\inv}(G)}}\right)\right],
\end{align*}
where $u_K = \rk \mathcal{O}_K^{\times} = r_1(K) + r_2(K) - 1 = |P_\infty| - 1$.
\end{enumerate}
\end{corollary}

The proof is a consequence of the Tauberian theorem \cite[Corollary 2.4]{alberts2021}. The discriminant is one such invariant covered by Corollary \ref{cor:MB_Taub}, with weight function given by the index function $\ind(g) = n - \#\{\text{orbits of }g\}$. Another important example is the product of ramified primes ordering, corresponding to the weight function that equals $1$ on all nonidentity elements. A family of local conditions $\Sigma$ satisfying the conditions of Corollary \ref{cor:MB_Taub} is one that makes a restriction at only finitely many places.

\section{Applying the Law of Large Numbers}\label{sec:proof}

We prove Theorem \ref{thm:muK_Malle_general} in this section, which is a slight generalization of Theorem \ref{thm:prob1intro} to include other orderings like the product of ramified primes ordering and restricted local condtions at finitely many places.

This will be done via the main results of \cite{alberts2022}. We will prove some facts about epi-products in ${\rm Grp}(K)$, which will allow us to bound the mixed moments of $\inv_X^{G,S,\Sigma}$ that appear in \cite[Theorem 1.4]{alberts2022} in order to complete step (\label{it:mixed_moments}). These, in conjunction with \cite[Theorem 1.3]{alberts2022}, will be used to prove Theorem \ref{thm:muK_Malle_general}.

\subsection{Counting functions ordered by admissible invariants}

The proof of Theorem \ref{thm:prob1intro} extends to the following result with little work. As it is natural to work at this level of generality, we will prove the following:

\begin{theorem}\label{thm:muK_Malle_general}
Let $K$ be a number field, $S$ a finite set of places, $G$ a finite group, $\inv$ an admissible invariant for which there exists a weight function $w:G\to \Z_{\ge 0}$ such that
\begin{itemize}
\item{$w(1) = 0$,}
\item{$w$ is constant on $K$-conjugacy classes of $G$, and}
\item{for all but finitely many places $p$, $\inv_p(\pi) = w(g)$ if and only if $\pi(I_p) = \langle g\rangle$,}
\end{itemize}
and $\Sigma = (\Sigma_p)$ a family of local conditions for which $\Sigma_p = \Hom(G_{K_p},G)$ at all but finitely many places of $K$. Then
\begin{enumerate}
\item[(i)] For any $\epsilon > 0$,
\[
\frac{N(\mathscr{G},\inv_X^{G,S,\Sigma})}{X^{1/a_{\inv}(G)+\epsilon}} \overset{a.s.}{\longrightarrow} 0
\]
as $X\to \infty$, where the ``a.s." stands for ``converges almost surely".
\item[(ii)] If $G = \langle g\in G : w(g) = a_{\inv}(G)\rangle$ is generated by minimal weight elements, then
\[
\frac{N(\mathscr{G},\inv_X^{G,S,})}{c_{\inv}(K,\Sigma)X^{1/a_{\inv}(G)}(\log X)^{b_{\inv}(K,G) - 1}} \overset{p.}{\longrightarrow} 1
\]
as $X\to \infty$, where the ``p." stands for ``converges in probability".
\item[(iii)] Suppose every proper normal subgroup $N\normal G$ satisfies one of the following:
\begin{enumerate}
\item[(a)] $N$ contains no elements of minimal weight, or
\item[(b)] $G\setminus N$ contains at least two $K$-conjugacy classes of minimal weight.
\end{enumerate}
Then
\[
\frac{N(\mathscr{G},\inv_X^{G,S,\Sigma})}{c_{\inv}(K,\Sigma)X^{1/a_{\inv}(G)}(\log X)^{b_{\inv}(K,G) - 1}} \overset{a.s.}{\longrightarrow} 1
\]
as $X\to \infty$, where the ``a.s." stands for ``converges almost surely".
\end{enumerate}
The invariants $a_{\inv}(G)$, $b_{\inv}(K,G)$, and $c_{\inv}(K,\Sigma)$ are defined as in Corollary \ref{cor:MB_Taub}.
\end{theorem}

Certainly the discriminant ordering satisfies these conditions, so it is clear that Theorem \ref{thm:prob1intro} is a special case of Theorem \ref{thm:muK_Malle_general}. Of particular interest is the product of ramified primes ordering, which also satisfies these conditions. The remainder of this section will be focused on proving Theorem \ref{thm:muK_Malle_general}.

\subsection{Epi-products in ${\rm Grp}(K)$}

The categorical results of \cite{alberts2022} rely entirely on the existence or nonexistence of epi-products respected by the finite moments $M$. The epi-product of objects $G_1$ and $G_2$ is defined in \cite[Definition 3.1]{alberts2022} as an object $G_1\times_{\rm epi} G_2$ satisfying the universal property
\[
\begin{tikzcd}
H \dar[two heads]\drar[two heads, dashed]\rar[two heads] &G_2\\
G_1 &G_1\times_{\rm epi} G_2\uar[two heads]\lar[two heads]
\end{tikzcd}
\]
where all morphisms in the diagram (including the universal morphism) are required to be epimorphisms. The existence of epi-products is a property of the category, and not dependent on the ordering being considered, so we prove a classification of epi-products in ${\rm Grp}(K)$ separately.

The category ${\rm Grp}(K)$ has finite products between objects with local data at the same places given by
\[
(G_1,S,\phi_1)\times (G_2,S,\phi_2) = (G_1\times G_2,S,\phi_1\times\phi_2),
\]
which is inherited from the product structure on ${\rm Grp}$. If an epi-product exists, it must be isomorphic to the usual product via the universal morphism.

\begin{lemma}\label{lem:epiprod}
The objects $(G_1,S,\phi_1)$ and $(G_2,S,\phi_2)$ have an epi-product in ${\rm Grp}(K)$ if
\[
G_1 \times G_2 = \langle (\phi_1\times \phi_2)(G_{K_p}): p\in S \rangle.
\]
\end{lemma}

The existence of epi-products is closely tied to the correlation of $\#{\rm Epi}(\mathscr{G},(G_1,S,\phi_1))$ and $\#{\rm Epi}(\mathscr{G},(G_2,S,\phi_2))$. These are constructed so that elements $\pi_i\in{\rm Epi}(\mathscr{G},(G_i,S,\phi_i))$ model number fields $L_1/K$ and $L_2/K$ with Galois groups $G_1$ and $G_2$ respectively with prescribed local conditions given by $\phi_1$ and $\phi_2$. The condition that
\[
G_1 \times G_2 = \langle (\phi_1\times \phi_2)(G_{K_p}) : p\in S\rangle
\]
translates to $L_1\cap L_2 = K$, i.e. $L_1$ and $L_2$ have trivial intersection. By the Chebotarev density theorem, this is equivalent to the distributions of Frobenius in $\Gal(L_1/K)$ and $\Gal(L_2/K)$ being independent. As $\mu_K^{\rm MB}$ and $\mu_{K,S}^{\rm MB}$ were constructed from the distribution of Frobenius elements given by Chebotarev density, intuitively it stands to reason that independence of these distributions corresponds to independence of $\pi_1$ and $\pi_2$ in some appropriate sense.

The author proves that $\#{\rm Epi}(\mathscr{G},(G_1,S,\phi_1))$ and $\#{\rm Epi}(\mathscr{G},(G_2,S,\phi_2))$ are uncorrelated when the epi-product exists (and $M$ respects the product structure) in greater generality in \cite[Lemma 5.2]{alberts2022}. This proof is done more directly, but the author's primary inspiration for those results is precisely this correspondence.

\begin{proof}
Suppose we have a commutative diagram
\[
\begin{tikzcd}
(H,S',\psi) \rar[two heads] \dar[two heads] \drar[dashed]{\pi}&(G_2,S,\phi_2)\\
(G_1,S,\phi_1)  & (G_1\times G_2,S,\phi_1\times\phi_2)\lar[two heads]\uar[two heads],
\end{tikzcd}
\]
where $\pi$ is the universal morphism, and suppose that $\langle (\phi_1\times \phi_2)(G_{K_p}) \rangle = G$. Then
\[
G_1\times G_2 = \langle \pi\psi(G_{K_p})\rangle \le \im \pi \le G_1\times G_2,
\]
which implies $\pi$ is surjective as a group homomorphism and therefore is an epimorphism. This is independent of the choice of $(H,S',\psi)$, and thus proves the existence of the epi-product.
\end{proof}

As described in \cite[Theorems 1.4 and 1.5]{alberts2022}, proving the Law of Large Numbers for a counting function depends only on those tuples of elements for which an epi-product does not exist. By showing that these objects have density zero under the ordering, \cite[Theorem 1.3]{alberts2022} will give the asymptotic growth rate with probability $1$. We can use Lemma \ref{lem:epiprod} to give a description of these objects.

As in \cite[Theorem 1.4]{alberts2022}, we define $E(2,M)$ to be the set of objects $((G_1,S_1,\phi_1),(G_2,S_2,\phi_2))\in{\rm Grp}(K)^2$ for which
\begin{itemize}
\item The epi-product $(G_1,S_1,\phi_1)\times_{\rm epi}(G_2,S_2,\phi_2)$ exists, and
\item $M_{(G_1,S_1,\phi_1)\times_{\rm epi}(G_2,S_2,\phi_2)} = M_{(G_1,S_1,\phi_1)} M_{(G_2,S_2,\phi_2)}$.
\end{itemize}

\begin{lemma}\label{lem:nonepiprod}
Let $M$ be the discrete measure on ${\rm Grp}(K)/\cong$ given by $M(G,S,\phi) = M_{(G,S,\phi)} = |G^{\rm ab}[|\mu(K)|]|^{-1}|G|^{-|S\cup S_\infty|+1}$, and let
\[
((G_1,S,\phi_1),(G_2,S,\phi_2))\in {\rm Grp}(K)^{2}\setminus E(2,M)
\]
with the set of places $S$ being the same at each coordinate. Then the group
\[
D=\langle \left(\phi_{1,p}\times \phi_{2,p}\right)(G_{K_p}) : p\in S\rangle
\]
is a proper subgroup of $G_1\times G_{2}$ for which at least one of the following is true:
\begin{enumerate}
\item[(a)] $\rho_i(D) \ne G_i$ for at least one $i\in\{1,2\}$, where $\rho_i:G_1\times G_{2}\to G_i$ is projection onto the $i^{\rm th}$ coordinate, or
\item[(b)] $\iota_j(G_j)\not\subseteq D$ for all $j\in\{1,2\}$, where $\iota_j:G_j \to G_1\times G_{2}$ is the inclusion morphism into the $j^{\rm th}$ coordinate.
\end{enumerate}
\end{lemma}

\begin{proof}
We prove this result by contrapositive, so we suppose that $\rho_m(D) = G_m$ for each of $m\in \{1,2\}$ and that there exists a coordinate $j$ for which $\iota_j(G_j) \subseteq D$. Without loss of generality, suppose $j=1$. Then $G_1\times 1\subseteq D$ and $\rho_2(D) = G_2$ implies $G_1\times G_2\subseteq D$. By Lemma \ref{lem:epiprod}, this implies $(G_1,S,\phi_1)$ and $(G_2,S,\phi_2)$ has an epi-product, which is necessarily isomorphic to the direct product $(G_1,S,\phi_1)\times (G_2,S,\phi_2)$.

Moreover,
\begin{align*}
M_{(G_1,S,\phi_1)\times (G_2,S,\phi_2)} &= |(G_1\times G_2)^{\rm ab}[|\mu(K)|]|^{-1}|G_1\times G_2|^{-|S\cup P_\infty|+1}\\
&= |G_1^{\rm ab}[|\mu(K)|]|^{-1}|G_1|^{-|S\cup P_\infty|+1}|G_2^{\rm ab}[|\mu(K)|]|^{-1}|G_2|^{-|S\cup P_\infty|+1}\\
&= M_{(G_1,S,\phi_1)}M_{(G_2,S,\phi_2)}.
\end{align*}
By definition, this implies 
\[
((G_1,S,\phi_1),(G_2,S,\phi_2))\in E(2,M).
\]
Thus, we have proven the result by contrapositive.
\end{proof}

\subsection{Bounding the mixed moments}\label{subsec:mixed}

We will be proving Theorem \ref{thm:muK_Malle_general} by appealing to \cite[Theorems 1.3 and 1.4]{alberts2022}. This requires bounds on the mixed moments of $\inv_X^{G,S,\Sigma}$ in order to apply these result. Given a probability measure $\mu$ on the category of pro-objects of $C$ with finite moments $M_G$, the mixed moments $M^{(j)}:C^j/\cong \to [0,\infty]$ are defined by
\[
M^{(j)}_{(G_1,G_2,...,G_j)} = \int \left(\prod_{i=1}^j \#{\rm Epi}(\mathscr{G},G_i)\right)\ d\mu(\mathscr{G}).
\]
The mixed moments need not be finite in general, but we will prove that they are in the cases of interest to us.

\begin{lemma}\label{lem:mixedmoment}
Let $M_{(G,S,\phi)} = |G^{\rm ab}[|\mu(K)|]|\cdot|G|^{-|S\cup P_\infty|+1}$ be the finite moments of $\mu_K^{\rm MB}$. Fix $K$, $G$, $S$, $\inv$, and $\Sigma$ as in Theorem \ref{thm:muK_Malle_general}. Then
\begin{enumerate}
\item[(a)] For each positive integer $k$ and each $j\in\{1,2,...,2k\}$,
\[
\int_{{\rm Grp}(K)^{2k}} \inv_X^{G,S,\Sigma}\ dM^{(j)}(dM)^{2k-j} \ll X^{\frac{2k}{a_{\inv}(G)} + \epsilon}.
\]
\item[(b)] For each integer $j\in\{1,2\}$,
\[
\int_{{\rm Grp}(K)^{2}\setminus E(2,M)} \inv_X^{G,S,\Sigma}\ dM^{(j)}(dM)^{2-j} \ll X^{\frac{2}{a_{\inv}(G)}}(\log X) ^{2b_{\inv}(K,G)-2\beta-1},
\]
where
\[
\beta= \min\left\{\#\left\{\substack{\displaystyle K\text{-conjugacy classes }\kappa\subseteq G\setminus N\\\displaystyle\text{with }w(\kappa)=a_{\inv}(G)}\right\} \Bigg{|}\substack{\displaystyle N\normal G\text{ a proper normal subgroup}\\\displaystyle\text{ with }N\text{ containing at least one }\\\displaystyle \text{minimal weight element.}}\right\}.
\]
\end{enumerate}
Moreover, if $G = \langle g\in G : w(g) = a_{\inv}(G)\rangle$ is generated by minimal weight elements then $\beta \ge 1$.
\end{lemma}

Similar results to Lemma \ref{lem:mixedmoment}(b) exist for the $2k^{\rm th}$ mixed moments, in line with \cite[Theorem 1.5]{alberts2022}, but they will not be necessary for the proof of Theorem \ref{thm:muK_Malle_general}.

Lemma \ref{lem:mixedmoment} is the most technical result of this paper, but is ultimately a consequence of Lemma \ref{lem:nonepiprod}. The primary technique is to use \cite[Lemma 5.3]{alberts2022} to bound the mixed moments above by a sum of first moments. The content of \cite[Lemma 5.3]{alberts2022} is as follows: Let $C$ be a category in which every morphism decomposes uniquely (up to isomorphism) as the composition of an epimorphism with a monomorphism. If $G_1,...,G_j$ are objects in $C$ for which $G_1\times\cdots \times G_{j}$ exists and has finitely many subobjects, then
\[
M_{(G_1,G_2,...,G_j)}^{(j)} = \sum_{\substack{\iota:H\hookrightarrow G_1\times \cdots \times G_j\\ \rho_i\iota\text{ is an eipmorphism}}} M_H,
\]
where the sum is over all subobjects $H$ on which each projection map restricts to an epimorphism on $H$. In particular, this implies $M_{(G_1,G_2,...,G_j)}^{(j)}<\infty$.

Many objects in the category ${\rm Grp}(K)$ have these properties.
\begin{itemize}
\item The product between elements with local conditions at the same set of places is given by
\[
(G,S,\phi) \times (H,S,\varphi) = (G\times H, S, \phi\times \varphi).
\]
The product need not exist when the sets of places differ, but this will be sufficient for our purposes.
\item A morphism $\pi:(G,S,\phi) \to (H,S',\varphi)$ decomposes as
\[
(G,S,\phi) \twoheadrightarrow (\im(\pi), S', \varphi) \hookrightarrow (H,S',\varphi),
\]
in the same way that morphisms of groups decompose.
\end{itemize}

Using this result, we will bound the quantities in Lemma \ref{lem:mixedmoment} above by the integrals of different admissible invariants. Lemma \ref{lem:originalmoment} and Corollary \ref{cor:MB_Taub} will then give the asymptotic bounds. We get the stronger bounds in part (b) by showing that avoiding objects in $E(2,M)$ can be translated into avoiding certain ramification behaviors.

\begin{proof}[Proof of Lemma \ref{lem:mixedmoment}(a)]
Firstly, we remark that
\begin{align*}
\int_{{\rm Grp}(K)^{2k}} \inv_X^{G,S,\Sigma}\ dM^{(j)}(dM)^{2k-j} &= \left(\int_{{\rm Grp}(K)^{j}} \inv_X^{G,S,\Sigma}\ dM^{(j)}\right)\left(\int_{{\rm Grp}(K)} \inv_X^{G,S,\Sigma}\ dM\right)^{2k-j}.
\end{align*}
Lemma \ref{lem:originalmoment} and Corollary \ref{cor:MB_Taub} bounds
\[
\int_{{\rm Grp}(K)} \inv_X^{G,S,\Sigma}\ dM \ll X^{\frac{1}{a_{\inv}(C)} + \epsilon},
\]
so it suffices to consider the $M^{(j)}$ factor on its own. Given that $\inv_X^{G,S,\Sigma}$ is supported on objects of the form $(G,S(X),\phi)$ for $S(X) = S\cup \{|p|\le X\}$, we can write the integral as
\begin{align*}
\int_{{\rm Grp}(K)^{j}} \inv_X^{G,S,\Sigma}\ dM^{(j)} &= \sum_{(G,S(X),\phi_i)\in {\rm Grp}(K)^{j}} \left(\prod_{i=1}^{j}\inv_X^{G,S,\Sigma}(G,S(X),\phi_i)\right) M_{((G,S(X),\phi_1),...,(G,S(X),\phi_j))}^{(j)}.
\end{align*}
Write $\varphi=\phi_1\times \cdots \times\phi_j$ so that $\rho_i\varphi = \phi_i$. We then use \cite[Lemma 5.3]{alberts2022} to bound the mixed moment $M^{(j)}$ by
\begin{align*}
M_{((G,S(X),\rho_1\varphi),...,(G,S(X),\rho_j\varphi))}^{(j)} &= \sum_{\substack{(H,S(X),\psi)\le (G^{j},S(X),\varphi)\\\forall i, \rho_i(H) = G}} M_{(H,S(X),\psi)}\\
&=\sum_{\substack{H\le G^{j}\\\im\left(\varphi\right)\le H\\\forall i, \rho_i(H) = G}} M_{(H,S(X),\varphi)}
\end{align*}
For each $H\le G^j$ with $\rho_i(H) =G$ for all $i=1,2,...,j$, the universal property of direct products gives a bijection between the objects $(G,S(X),\phi_i)\in {\rm Grp}(K)^{j}$ for which $\im\left(\prod_{i=1}^j \phi_i\right)\le H$ and the set of objects $(H,S(X),\varphi)\in {\rm Grp}(K)$. Thus, we can pull the sum over $H\le G^{j}$ to the outside to produce
\begin{align*}
&\sum_{\substack{H\le G^{j}\\\forall i,\rho_i(H) = G}}\sum_{(H,S(X),\varphi)\in {\rm Grp}(K)} \left(\prod_{i=1}^{j}\inv_X^{G,S,\Sigma}(G,S(X),\rho_i\varphi)\right)M_{(H,S(X),\varphi)}\\
&=\sum_{\substack{H\le G^{j}\\\forall i,\rho_i(H) = G}}\sum_{\substack{\varphi:F_{K,S(X)} \to H\\\forall i,\ \rho_i\varphi\in \Sigma\\\forall i,\ |\inv(\rho_i\varphi)| \le X}} M_{(H,S(X),\varphi)}.
\end{align*}
Given that there are $|H|$ possible unramified maps $\phi_p:G_{K_p}\to H$ for each finite place $p$, we can reframe the finite moments for a larger set of places
\begin{align*}
M_{(H,S(X),\varphi)} &= |H^{\rm ab}[|\mu(K)|]|^{-1}|H|^{-|S(X)|+1}\\
&=|H|^{|S(X^{j})| - |S(X)|}|H^{\rm ab}[|\mu(K)|]|^{-1}|H|^{-|S(X^{j})|+1}\\
&=\sum_{\substack{\psi:F_{K,S(X^{j})} \to H\\\psi\text{ unram at }p\not\in S(X)\\\psi|_{S(X)}=\varphi}}M_{(H,S(X^{j}),\psi)}.
\end{align*}
Plugging his into our integral, we can bound
\begin{align*}
&\sum_{\substack{H\le G^{j}\\\forall i,\rho_i(H) = G}}\sum_{\substack{\varphi:F_{K,S(X)} \to H\\\forall i,\ \rho_i\varphi\in \Sigma\\\forall i,\ |\inv(\rho_i\varphi)| \le X}} \sum_{\substack{\psi:F_{K,S(X^{j})} \to H\\\psi\text{ unram at }p\not\in S(X)\\\psi|_{S(X)}=\varphi}}M_{(H,S(X^{j}),\psi)}\\
&=\sum_{\substack{H\le G^{j}\\\forall i,\rho_i(H) = G}}\sum_{\substack{\psi:F_{K,S(X^{j})} \to H\\\forall i,\ \rho_i\psi\in \Sigma\\\forall i,\ |\inv(\rho_i\psi)| \le X}} M_{(H,S(X^{j}),\psi)}\\
&\le \sum_{\substack{H\le G^{j}\\\forall i,\rho_i(H) = G}}\sum_{\substack{\psi:F_{K,S(X^{j})} \to H\\\forall i,\ \rho_i\psi\in \Sigma\\\ \prod_{i=1}^{j}|\inv(\rho_i\psi)| \le X^{j}}} M_{(H,S(X^{j}),\psi)}.
\end{align*}
For each $H\le G^j$, define an invariant $\inv_H:\prod_p \Hom(G_{K_p},H) \to I_K$ by
\[
\inv_H(\pi) = \prod_{i=1}^{j} \inv(\rho_i(\pi)).
\]
This is an admissible invariant with weight function $w_H(g_1,...,g_{j}) = w(g_1) + \cdots +w(g_{j})$, inheriting the necessary properties from $\inv$. We additionally define $\Sigma^{j} = (\Sigma_p^j)$ for the family of local conditions $\Sigma_p^j\subseteq\Hom(G_{K_p},H)$ determined by $\rho_i(\Sigma_p^j) = \Sigma_p$ for projection to each coordinate $i=1,...,j$. The sum simplifies then simplifies to 
\begin{align*}
& \sum_{\substack{H\le G^{j}\\\forall i,\rho_i(H) = G}}\sum_{(H,S(X^{j}),\psi)\in {\rm Grp}(K)} (\inv_H)_{X^{j}}^{H,S,\Sigma^j}(H,S(X^{j}),\psi) M_{(H,S(X^{j}),\psi)}\\
&= \sum_{\substack{H\le G^{j}\\\forall i,\rho_i(H) = G}} \int_{{\rm Grp}(K)} (\inv_H)_{X^{j}}^{H,S,\Sigma^j}\ dM.
\end{align*}
The result then follows from Lemma \ref{lem:originalmoment} and Corollary \ref{cor:MB_Taub}, noting that the minimal weight necessarily satisfies
\[
a_{\inv_H}(H) \ge a_{\inv}(G)
\]
and that, in the context of Corollary \ref{cor:MB_Taub}, $C = G$ for the family $\Sigma^j$. Thus
\begin{align*}
\int_{{\rm Grp}(K)^{j}} \inv_X^{G,S,\Sigma^j}\ dM^{(j)} &\le \sum_{\substack{H\le G^{j}\\\forall i,\rho_i(H) = G}} \int_{{\rm Grp}(K)} (\inv_H)_{X^{j}}^{H,S,\Sigma^j}\ dM\\
&\ll \sum_{\substack{H\le G^{j}\\\forall i,\rho_i(H) = G}}(X^{j})^{\frac{1}{a_{\inv}(G)} + \epsilon}\\
&\ll X^{\frac{j}{a_{\inv}(G)} + \epsilon}.
\end{align*}
\end{proof}

Lemma \ref{lem:mixedmoment}(b) is more stringent, and so we cannot use as loose of bounds as in part (a). Instead, Lemma \ref{lem:nonepiprod} will play a major role in controlling the number of minimum index elements that can appear in the local data.
\begin{proof}[Proof of Lemma \ref{lem:mixedmoment}(b)]
By Lemma \ref{lem:nonepiprod}, we know that $((G_1,S,\phi_1),(G_2,S,\phi_2))\in {\rm Grp}(K)^{2}\setminus E(2,M)$ implies that the group
\[
D = \langle (\phi_1\times \phi_2)(G_{K_p}):p\in S\rangle
\]
either has one of $\rho_i(D)\ne G_i$, or $\iota_i(G_i) \not\subseteq D$ for each $i$. We will use this to refine the proof of part (a) to bound the integral.

Suppose first that $j=2$. Similar to part (a), we compute
\begin{align*}
\int_{{\rm Grp}(K)^{2}\setminus E(2,M)} \inv_X^{G,S,\Sigma}\ dM^{(2)}&\le \sum_{\substack{H\le G^2\\\forall i,\ \rho_i(H) =G}}\sideset{}{^*}\sum_{D\le G^2}\sum_{\substack{(H,S(X),\varphi)\in {\rm Grp}(K)\\ \im(\varphi)\le D}} \prod_{i=1}^2\inv_X^{G,S,\Sigma}(G,S(X),\rho_i\varphi)M_{(H,S(X),\varphi)}\\
&=\sum_{\substack{H\le G^2\\\forall i,\ \rho_i(H) =G}}\sideset{}{^*}\sum_{D\le G^2}\sum_{\substack{\varphi:F_{K,S(X)} \to H\\\im(\varphi)\le D\\ \varphi\in \Sigma\\ |\inv(\rho_i\varphi)| \le X}} M_{(H,S(X),\rho_i\varphi)},
\end{align*}
where the $*$ on the middle sum indicates it is over those $D$ described by Lemma \ref{lem:nonepiprod}. We bound this via each of the following facts:
\begin{itemize}
\item We know that
\begin{align*}
M_{(H,S(X),\varphi)}M_{(G,S(X),\rho_2\varphi)} &= \sum_{\substack{\psi:F_{K,S(X^2)} \to H\\ \psi\text{ unram at }p\not\in S(X)\\ \psi|_{S(X)} = \varphi}} M_{(H,S(X^2),\psi)}
\end{align*}
similar to the equality used in part (a).
\item Let $\Sigma_{H,D} = (\Sigma_{H,D,p})$ be the family of local conditions $\Sigma_{H,D,p}\subseteq \Hom(G_{K_p},H)$ for which
\[
\Sigma_{H,D,p} = \{f\in \Sigma_p^2\cap \Hom(G_{K_p},H) : f(I_p)\le D\}.
\]
\item The invariant $\inv_{H}:\prod_p\Hom(G_{K_p},H) \to I_K$ defined as in part (a) by
\[
\inv_{H}(\pi) = \inv(\rho_1\pi)\inv(\rho_2\pi)
\]
is necessarily admissible, inheriting the property from $\inv$, with weight function $w_{H}(g_1,g_2) = w(g_1)+w(g_2)$. 
\end{itemize}
Putting these together as in part (a) gives the bound
\begin{align*}
\int_{{\rm Grp}(K)^{2}\setminus E(2,M)} \inv_X^{G,S,\Sigma}\ (dM)^2&\le \sum_{\substack{H\le G^2\\\forall i,\ \rho_i(H) =G}}\sideset{}{^*}\sum_{D\le G^2}\sum_{\substack{\psi:F_{K,S(X^2)} \to H\\\psi\in \Sigma_{H,D}\\ |\inv_{H}(\psi)| \le X^2}} M_{(H,S(X^2),\psi)}\\
&=\sum_{\substack{H\le G^2\\\forall i,\ \rho_i(H) =G}}\sideset{}{^*}\sum_{D\le G^2}\sum_{(H,S(X^2),\psi)\in {\rm Grp}(K)} (\inv_{H})_{X^2}^{H,S,\Sigma_{H,D}}(H,S(X^2),\psi) M_{(H,S(X^2),\psi)}\\
&=\sum_{\substack{H\le G^2\\\forall i,\ \rho_i(H) =G}}\sideset{}{^*}\sum_{D\le G^2}\int_{{\rm Grp}(K)} (\inv_{H})_{X^2}^{H,S,\Sigma_{H,D}}\ dM.
\end{align*}

We consider the summands separately. Suppose first that $H\not\subseteq D$. Then $\Sigma_{H,D}$ is necessarily nonadmissible as $\Sigma_{H,D,p}\subseteq \Hom(G_{K_p},D)$ while $\Hom_{\rm ur}(G_{K_p},H)\not\subseteq \Hom(G_{K_p},D)$ for each place $p$. Thus, Lemma \ref{lem:originalmoment}(b) implies that each such summand is bounded by a decaying function
\[
\int_{{\rm Grp}(K)} (\inv_{H})_{X^2}^{H,S,\Sigma_{H,D}}\ dM = O\left(r^{-X/\log X}\right) = o(1).
\]

Now suppose that $H\subseteq D$. Then $\Sigma_{H,D,p} = \Hom(G_{K_p},H)$ for all but finitely many places $p$ imposes no local conditions at all. To indicate this, we write $\inv_{X^2}^{H,S,\Sigma_{H,D}}=\inv_{X^2}^{H,S}$. We now apply Lemma \ref{lem:originalmoment} and Corollary \ref{cor:MB_Taub}. As in part (a), $a_{\inv_{H}}(H) \ge a_{\inv}(G)$ by construction. If this inequality is strict, then we are done.

If $a_{\inv_H}(H) = a_{\inv}(G)$, we need to control the power of $\log X$. The only nonidentity elements of $H\le G^2$ that can have minimal weight agreeing with the minimal weight of $G$ are those of the form $(g,1)$ or $(1,g)$ for $w(g) = a_{\inv}(G)$, because the weight function is given by $w_H(g_1,g_2) = w(g_1) + w(g_2)$.

Let $H_1\le G$ be the subgroup such that $H_1\times 1 = \ker \rho_2|_{H} = \iota_1(G) \cap H$, and similarly for $H_2\le G$. That these groups are given by kernels implies $H_i\normal G$. Lemma \ref{lem:nonepiprod} together with the containment $H\subseteq D$ implies $H_i \subseteq \iota_i^{-1}(D)\ne G$, so they are proper normal subgroups. We then compute
\begin{align*}
b_{\inv_{H}}(K,H) &= \#\{K\text{-conjugacy classes in }H\text{ of minimal weight}\}\\
&=\sum_{i=1}^2 \#\{K\text{-conjugacy classes in }H_i\text{ of minimal weight}\}\\
&\le 2(b_{\inv}(K,G) - \beta)\\
&= 2b_{\inv}(K,G) - 2\beta.
\end{align*}
Part (b) for $j=2$ then follows. The $j=1$ case proceeds via the same argument, where only the summand $H = G^2$ occurs in the bound.

If we suppose that $G = \langle g\in G : w(g) = a_{\inv}(G)\rangle$, any proper normal subgroup $N\normal G$ cannot contain all the minimal weight elements. Thus, $G\setminus N$ contain at least one minimal weight element. Additionally, the fact that $N$ is a normal subgroup implies that both $N$ and $G\setminus N$ are closed under conjugation and invertible powers, and so are unions of $K$-conjugacy classes. Thus. there exists at least one $K$-conjugacy class of minimal weight in $G\setminus N$. As $N$ was arbitrary, this implies $\beta \ge 1$.
\end{proof}

\subsection{Proving the counting results}
With this result in hand, we can now prove Theorem \ref{thm:muK_Malle_general} (of which Theorem \ref{thm:prob1intro} is a special case).

\begin{proof}[Proof of Theorem \ref{thm:muK_Malle_general}]
The proof of Theorem \ref{thm:muK_Malle_general}(i) follows from \cite[Corollary 1.4]{alberts2022} and Lemma \ref{lem:mixedmoment}(a). \cite[Corollary 1.4]{alberts2022} states that
\[
N(\mathscr{G},\inv_X^{G,S,\Sigma}) = o\left(\max_{0\le j\le k}X^{\frac{1+\epsilon}{2k}}\left(\int_{{\rm Grp}(K)^{2k}} \inv_X^{G,S,\Sigma}\ dM^{(j)}(dM)^{2k-j}\right)^{1/2k}\right).
\]
almost surely. Applying the bound Lemma\ref{lem:mixedmoment}(a), this implies
\[
N(\mathscr{G},\inv_X^{G,S,\Sigma}) = o\left(\max_{0\le j\le k}X^{\frac{1+\epsilon}{2k} + \frac{1}{a_{\inv}(G)} + \epsilon}\right)
\]
almost surely. Replacing $\epsilon$ with $\epsilon/2$ and taking $k$ sufficiently large concludes the proof.

We can prove Theorem \ref{thm:muK_Malle_general}(ii,iii) by \cite[Theorem 1.3(ii,iii)]{alberts2022} respectively in conjunction with \cite[Theorem 1.4]{alberts2022}. \cite[Theorem 1.4]{alberts2022} states that
\[
\int_{{\rm proGrp}(K)} \left\lvert N(\mathscr{G},\inv_X^{G,S,\Sigma}) - \int_{{\rm Grp}(K)} \inv_X^{G,S,\Sigma}\ dM\right\rvert^2 d\mu_K^{\rm MB}(\mathscr{G}) \ll \max_{j\in \{1,2\}}\int_{{\rm Grp}(K)^2\setminus E(2,M)} \inv_X^{G,S,\Sigma}\ dM^{(j)}(dM)^{2-j},
\]
which by Lemma \ref{lem:mixedmoment} is bounded by
\[
X^{1/a_{\inv}(G)}(\log X)^{2b_{\inv}(K,G) - 2\beta - 1}.
\]
If $G$ is generated by mimimal weight elements, we necessarily have $\beta \ge 1$.

Thus, by Corollary \ref{cor:MB_Taub}
\[
\int_{{\rm proGrp}(K)} \left\lvert N(\mathscr{G},\inv_X^{G,S,\Sigma}) - \int_{{\rm Grp}(K)} \inv_X^{G,S,\Sigma}\ dM\right\rvert^2 d\mu_K^{\rm MB}(\mathscr{G}) \ll \frac{1}{(\log X)^{2\beta - 1}} \left(\int_{{\rm Grp}(K)} \inv_X^{G,S,\Sigma}\ dM\right)^2.
\]
It follows from $\beta \ge 1$ that $\frac{1}{(\log X)^{2\beta - 1}}\le \frac{1}{(\log X)} \to 0$. Thus \cite[Theorem 1.3(ii)]{alberts2022} directly proves Theorem \ref{thm:muK_Malle_general}(ii).

In the case that all proper normal subgroups $N\normal G$ either contain no minimal weight elements or have trivial intersection with at least two $K$-conjugacy classes of minimal weight, we necessarily have $\beta \ge 2$. Thus
\begin{align*}
\int_{{\rm proGrp}(K)} \left\lvert N(\mathscr{G},\inv_X^{G,S,\Sigma}) - \int_{{\rm Grp}(K)} \inv_X^{G,S,\Sigma}\ dM\right\rvert^2 d\mu_K^{\rm MB}(\mathscr{G}) &\ll \frac{1}{(\log X)^3} \left(\int_{{\rm Grp}(K)} \inv_X^{G,S,\Sigma}\ dM\right)^2\\
&\ll \frac{\left(\int_{{\rm Grp}(K)} \inv_X^{G,S,\Sigma}\ dM\right)^2}{\left(\log \int_{{\rm Grp}(K)} \inv_X^{G,S,\Sigma}\ dM\right)^3}.
\end{align*}
The last inequality is true by $\int_{{\rm Grp}(K)} \inv_X^{G,S,\Sigma}\ dM\ll X^{1/a_{\inv}(G)+\epsilon}\ll X^{2}$ following from Lemma \ref{lem:originalmoment} and Corollary \ref{cor:MB_Taub}. The function $\psi(X) = (\log X)^3$ is nondecreasing and satisfies $\sum \frac{1}{n(\log n)^3} < \infty$, so \cite[Theorem 1.3(iii)]{alberts2022} directly proves Theorem \ref{thm:muK_Malle_general}(iii).
\end{proof}

\section{Making predictions for number field counting}\label{sec:knownissues}

Our results suggest that the Vast Counting Heuristic described by the author in \cite[Heuristic 1.7]{alberts2022} applies to Malle's conjecture, as long as we expect the absolute Galois group $G_K$ to be ``typical" among groups with local data. Of course, it is well known that Malle's conjecture is false as stated - Kl\"uners provided the first counter example in $C_3\wr C_2\subseteq S_6$ for which Malle's predicted $b$-invariant is too small \cite{kluners2005}. Kl\"uners' counter example is witnessing some atypical behavior for $G_\Q$ among groups with local data distributed according to $\mu_K^{\rm MB}$, specifically the behavior that $\Gal(\Q(\zeta_3)/\Q)$ is a quotient of $G_\Q$.

Lemma \ref{lem:originalmoment} can be understood as recognizing $\mu_K^{\rm MB}$ as a random group version of the Malle-Bhargava heuristic \cite{bhargava2007,wood2019}. The Malle-Bhargava heuristic states that the growth rate of Malle's counting function can be read off the Malle-Bhargava local series
\[
\frac{|G|}{|G^{\rm ab}[|\mu(K)|]|}\prod_p \left(\frac{1}{|G|}\sum_{f\in \Hom(G_{K_p},G)} p^{-(\nu_p\disc(f))s}\right),
\]
with possibly some constant out front. The convergent Euler product representation for $c(K,G)$ can immediately be recognized as coming from this series. The Malle-Bhargava principle, like Malle's conjecture, is of course wrong in some cases. By analogy, this means we expect our random group with local data $\mu_K^{\rm MB}$ to miss certain behaviors and produce incorrect predictions in similar cases to the Malle-Bhargava principle (like with Kl\"uners' counter-example).

It would be useful to compile a list of known ``atypical" behaviors for the absolute Galois group, that is behaviors we know can break the Malle-Bhargava principle. To the author's knowledge, the following list represents the sources of all \emph{currently known} issues with the Malle-Bhargava principle (including the value of the leading constant):
\begin{enumerate}
\item Cases when $G$ is \textbf{not generated by minimal index elements}. Theorem \ref{thm:prob1intro} does not even cover such cases, and with good reason - atypical behavior tends to be exacerbated by such orderings. One sees this in \cite{cohen-diaz-y-diaz-olivier2002} for $D_4\subseteq S_4$ and \cite{fouvry-koymans2021} for ${\rm Heis}_3\subseteq S_9$, where the leading constant is given explicitely by a convergent infinite sum of Euler product rather than just one. Indeed, even for abelian groups the leading constant may be a finite sum of convergent Euler products \cite{wright1989,wood2009}.

\item \textbf{Property E}, as called in \cite{liu-wood-zureick-brown2019}, which states that a central embedding problem for $G_K$ is solvable if and only if the corresponding local embedding problems are solvable. This is not generically true for $\mu_K^{\rm MB}$, although it is known that this behavior is important for counting number fields. In class group statistics and related statistics of unramified objects, this property is related to the size of certain $|\mu(K)|$-torsion in the Schur multiplier of $G$ and can result in different constants out front. Some authors avoid this issue by considering groups for which the corresponding Schur multiplier is trivial \cite{liu-wood-zureick-brown2019,boston-bush-hajir2017}, while others have computed what the contribution of the Schur multiplier is expected to be \cite{rubinstein-salzedo2014,wood2019,sawin-wood2023} under various guises.

\item \textbf{Gaining extra roots of unity}. Kl\"uners' counter example shows that the $G$-extensions $L/K$ for which $\mu(L) \ne \mu(K)$ can accumulate beyond Malle's predicted growth rate \cite{kluners2005}. Malle's predictions are corrected for this behavior by T\"urkelli \cite{turkelli2015,alberts2021}, although the leading constant is not explicitly addressed in these corrections.

\item \textbf{The Grunwald problem}. Not every local condition can occur in a global extension of number fields. This was first noticed by Wang in correcting the Grunwald-Wang theorem \cite{wang1950}: there are no $C_8$-extensions of $\Q$ which are totally inert at $2$. Missing local conditions should result in missing terms in the leading Euler product, as seen for abelian groups in \cite{wood2009}. Identifying when the Grunwald problem has a positive solution is currently an open problem, see \cite{Neukirch1973,Harari2007,DLAN17,Motte18}.

\item \textbf{Fair versus unfair counting functions}. Even when Grunwald-Wang counter examples can be avoided, the leading constant for counting abelian extensions need not agree with the convergent Euler product in Theorem \ref{thm:prob1intro}. A result of Wood shows that these leading constants agree if the extensions are ordered by a so-called fair counting function \cite{wood2009} (Wood's main results say much more than this, but we restrict our attention only to the leading constant).

Wood defines the notion of a fair counting function in general for abelian extensions, but it is not immediately clear how this definition should generalize to nonabelian groups. The author personally considers this the most mysterious obstruction for counting nonabelian extensions. The primary example of a fair counting function on abelian groups is the product of ramified primes ordering. One expects this ordering to be fair for nonabelian groups as well, but the author is not aware of any proposals for a general definition of ``fair" in this context.
\end{enumerate}

We do not attempt to solve these issues in this paper. A complete fix accounting for any one of these behaviors would represent immense progress in the study of number field counting, and will be the subject of continued research of the author. We contend that one of the major difficultlies in surpassing these obstructions lies in their intersections. To some extent, each one is affected by the presence of roots of unity. It will often be the case that a finite group $G$ witnesses several of these behaviors simultaneously, making it difficult to study them as individuals.

Instead, we briefly discuss how $\mu_K^{\rm MB}$ interacts with these behaviors. This verifies long-standing beliefs in the field that the Malle-Bhargava principle for the product of ramified primes ordering should be ``essentially correct away from roots of unity."

\subsection{Agreement with known cases in the constant}

Theorem \ref{thm:muK_Malle_general} agrees with the counting results proven by Wood \cite{wood2009} for abelian extensions ordered by a fair counting function, as long as the Grunwald problem has a positive solution for all local restrictions for the Galois group. This includes all odd order abelian groups. Theorem \ref{thm:muK_Malle_general} otherwise disagrees with Wood's results, showing that Wood's counting results witness the ``Grunwald problem" and ``fair counting function" obstructions.

Moreover, Theorem \ref{thm:prob1intro} agrees with the asymptotic growth rate for $S_n$-extensions of $\Q$ for $n=3,4,5$, and the predicted growth rate for $n\ge 6$, down to the constant factor of $1/2$ in front of the Malle-Bhargava local series \cite{bhargava2007}. This comes from the factor
\[
\frac{1}{|S_n^{\rm ab}[|\mu(\Q)|]|} = \frac{1}{|C_2[2]|} = \frac{1}{2}
\]
appearing in the constant out front.

\subsection{$G$ not generated by minimal index elements}

This obstruction is known to have an affect on number field counting. The leading constant for counting $D_4\subseteq S_4$ extensions is a convergent sum of Euler products \cite{cohen-diaz-y-diaz-olivier2002}, which is significantly different than the leading constant in Theorem \ref{thm:prob1intro}. This type of behavior is common, occurring also for $C_2\wr H$-extensions \cite{kluners2012} and ${\rm Heis}_3\subseteq S_9$ \cite{fouvry-koymans2021}. The root cause of these issues is that, when $G$ is not generated by minimal index elements, there exist subfield that occur in a positive proportion of $G$-extensions. Despite changing the behavior of the leading constant, having a subfield occur in a positive proportion of $G$-extensions makes it \emph{easier} to determine the asymptotic growth rate. This is the subject of forthcoming work of the author with Lemke Oliver, Wang, and Wood, where the asymptotic growth rate is determined for Malle's counting function for a number of groups with are not generated by minimal index elements (including cases which agree \emph{or disagree} with Malle's Conjecture).

When $G$ is not generated by minimal index elements, one has $\beta = 0$ in Lemma \ref{lem:mixedmoment}(b). As a consequence, our current methods are not able to produce an asymptotic with probability $1$.

We claim that this should be expected for the random group with local data $\mu_K^{\rm MB}$ just from its structure. The Law of Large Numbers results in \cite{alberts2022} taking advantage of proving enough independence (or uncorrelatedness) between the random variables $\#{\rm Epi}(\mathscr{G},(G,S,\phi))$. In a sense, this type of ``independence" is a model for independence of the Frobenius distribution between two $G$-extensions $L_1/K$ and $L_2/K$. We know by Chebotarev density that the distributions of Frobenius in $L_1$ and $L_2$ are independent if and only if the extensions are disjoint, i.e. $L_1\cap L_2 = K$ is trivial. In this setting, \cite[Theorem 1.4]{alberts2022} can be interpreted as requiring that 100\% of pairs of $G$-extensions have trivial intersection $L_1\cap L_2 = K$.

Now, suppose $N\normal G$ is a proper normal subgroup of $G$ containing all the minimal index elements. For a fixed $G/N$-extension $F/K$, the twisted Malle's conjecture \cite{alberts2021} uses the Malle-Bhargava principle to predict that a positive proportion of $G$-extensions $L/K$ have fixed field $L^N = F$. This is in direct contradiction to what we need for \cite[Theorem 1.4]{alberts2022}, showing that there is no hope of Theorem \ref{thm:prob1intro} being true as stated in these cases.

This argument shows that, in some sense, $\mu_{K}^{\rm MB}$ is already seeing this obstruction. We require different methods to study the counting function in such cases, but one no longer expects a result that looks like Theorem \ref{thm:prob1intro}.

\subsection{Property E}

Property E was so named by Liu--Wood--Zureick-Brown in \cite{liu-wood-zureick-brown2019} where the ``E" stands for ``Extensions".  It is equivalent to a certain local-to-global property: given a central short exact sequence $1\to Z\to E\to G\to 1$, the embedding problem
\[
\begin{tikzcd}
&&&G_K \dar[two heads]{\pi}\dlar[dashed]\\
1 \rar & Z \rar & E \rar & G \rar & 1
\end{tikzcd}
\]
is solvable if and only if the corresponding local embedding problems
\[
\begin{tikzcd}
&&&G_{K_p} \dar[two heads]{\pi|_{G_{K_p}}}\dlar[dashed]\\
1 \rar & Z \rar & E \rar & G \rar & 1
\end{tikzcd}
\]
are solvable for each place $p$ of $K$. This forces the existence of certain $E$-extensions given the existence of certain $G$-extensions, clearly resulting in some effects for number field counting problems. No examples of this property affecting Malle's counting function specifically currently exist in the literature, but the author is aware of an example in forthcoming work of Koymans--Pagano where property E is used to show Malle's predicted power of $\log X$ for the product of ramified primes ordering is incorrect for a certain nilpotency class $2$ group.

By replacing the absolute Galois group $G_K$ with a group with local data $\mathscr{G}\in {\rm proGrp}(K)$ in the above diagrams, we can consider property E for groups with local data. Relating the Malle-Bhargava local series directly to central embedding problems is much trickier to formulate; by considering the random group with local data $\mu_K^{\rm MB}$ we make this relationship clear.

In particular, we can ask what the probability is that property E holds with respect to $\mu_K^{\rm MB}$. One can show that this event is measurable, but at this time we are not able to compute the actual probability. Given work  on statistics of unramified objects whose leading constants disagree with the Malle-Bhargava local series, such as \cite{boston-bush-hajir2017,liu-wood-zureick-brown2019,wood2019,sawin-wood2023}, we expect that property E has probability strictly less than $1$ with respect to $\mu_K^{\rm MB}$.

We will not, however, leave this subsection empty handed. It is worth noting at this time that $\mu_K^{\rm MB}$ ``essentially" specializes to the random $\Gamma$-group defined by Liu--Wood--Zureick-Brown \cite{liu-wood-zureick-brown2019}. We consider maximal quotients of $\mathscr{G}$ of the form $H\rtimes \Gamma$, in which all ramification lies outside of $H$ and $H$ is prime to $|\mu(K)||\Gamma|$. Given any $\pi:F_{K,S}\to \Gamma$, if $\pi$ factors through $\mathscr{F}_{K,S}(r)$ then the maximal unramified extension of $\pi$ of this form is given by the group
\begin{align*}
\mathscr{F}_{K,S}(r) / \langle \ker\pi|_{I_p}\rangle &= F_{K,S} / \langle r_i\in \ker\pi\ (i=1,...,|S|+u), \ker\pi|_{I_p}\rangle\\
&= \langle \Fr_p : p\in S\rangle_{|\mu(K)||\Gamma|} \rtimes \Gamma / \langle r_i\in \ker\pi\ (i=1,...,|S|+u)\rangle\\
&= \left(\langle \Fr_p : p\in S\rangle_{|\mu(K)||\Gamma|} / \langle r_i\gamma r_i^{-1}\gamma^{-1}\ (i=1,...,|S|+u\text{ and }\gamma\in \Gamma)\rangle\right)\rtimes \Gamma
\end{align*}
If $\pi$ corresponds to a $\Gamma$-extension $L/K$, then this group is exactly equal to the random group proposed in \cite{liu-wood-zureick-brown2019} (with a semidirect product by $\Gamma$). This is where we say these random groups are ``essentially" the same. The difference is that $\pi$ which factors through $\mathscr{F}_{K,S}(r)$ merely \emph{models} a $\Gamma$-extension rather than being equal to one on the nose. This is due to a difference in setting: our random group is constructed with the goal of counting extensions with restricted local behavior, so the entire $H\rtimes \Gamma$ extensions are being modeled here. Meanwhile, Liu--Wood--Zureick-Brown construct a random group to model just the unramified $H$ piece, and varying over actual $\Gamma$-extensions (not models of $\Gamma$-extensions).

It is notable that Liu--Wood--Zureick-Brown's special relations of the form $r\gamma r^{-1}\gamma^{-1}$ falls out of the inertia data we included in the definition of a group with local data. Liu--Wood--Zureick-Brown constructed this expression for their relations by imposing property E on their random group. This suggests that the inclusion of inertia data in $\mu_K^{\rm MB}$ captures at least some part of property E.

\subsection{Root of unity behavior}

Kl\"uners gave the example $C_3\wr C_2\subseteq S_6$ for which Malle's predicted growth rate was incorrect \cite{kluners2005}. The issue was not just in the constant term, but in fact Malle's predicted power of $\log X$ is incorrect for this group. The behavior of the counting functions changes in the presence of additional roots of unity. The $C_3\wr C_2$-extensions $L/\Q$ for which $\Q(\zeta_3)\subseteq L$ are asymptotically more abundant than other $C_3\wr C_2$-extensions. The root cause is that every prime ideal of $\Q(\zeta_3)$ either divides $3$ or is $1$ mod $3$, which is not true for other quadratic extensions, allowing for more ways to ramify in the $C_3$-extension on top.

This makes the extensions by roots of unity ``exceptional" in some sense, capturing some fundamental behavior of the absolute Galois group.  This behavior was not built into the random group with local data $\mu_{K}^{\rm MB}$, and we can in fact prove that root of unity behavior occurs with probability $0$.

\begin{definition}\label{def:awayfromL}
Let $L/K$ be an extension of number fields. The subcategory ${\rm proGrp}_{\neg L}(K)$ of groups with local data away from $L$ is defined to be the subcategory of ${\rm proGrp}(K)$ whose objects satisfy
\[
\Hom(\mathscr{G},(\Gal(E/K),\phi_{E/K})) = \emptyset
\]
for each subextension $E\le L$ with $K\ne E$.
\end{definition}

Given an extension $M/K$, the corresponding Galois group with local data belongs to ${\rm proGrp}_{\neg L}(K)$ if and only if $M\cap L = K$. Groups with local data that could be the source of a Kl\"uners-style counter example would be objects in ${\rm ProGrp}(K)\setminus {\rm proGrp}_{\neg K(\zeta^e)}(K)$ for some integer $e$. As the following theorem shows, these objects occur with probability $0$ in the distribution:

\begin{theorem}
Let $L/K$ be a nontrivial finite extension of number fields and $\mu_K^{\rm MB}$ the measure given by Theorem \ref{thm:muK_construct}. Then
\[
\mu_K^{\rm MB}\left({\rm proGrp}(K)\setminus {\rm proGrp}_{\neg L}(K)\right) = 0.
\]
\end{theorem}

This means that taking probabilities conditional on ${\rm proGrp}_{\neg L}(K)$ acts as if we had taken no conditions at all. Thus, we can avoid any finite extension $L/K$ we want without losing any information from the distribution $\mu_K^{\rm MB}$. This behavior is not shared by the absolute Galois group; counting all $G$-extensions $L/K$ versus only those $G$-extensions for which $L\cap K(\zeta^e) = K$ can give different asymptotic growth rates by affecting the power of $\log X$. This result suggests that the random group with local data $\mu_K^{\rm MB}$ would be a better fit for counting the latter extensions.

\begin{proof}
Any element $(G,\phi)\in {\rm proGrp}(K)\setminus {\rm proGrp}_{\neg L}(K)$ has a morphism to $(\Gal(E/K),\phi_{E/K})$ for some $E\le L$, $E\ne K$, and as the conjugates of $\phi_{E/K}(G_{K_p})$ generate $\Gal(E/K)$ it must be that the morphism is an epimorphism. Thus, it is given by some surjective group homomorphism $\pi:G\to \Gal(E/K)$ for which $\pi\phi_p=\phi_{E/K,p}$.

Thus, it follows that
\begin{align*}
\int_{{\rm proGrp}(K)\setminus {\rm proGrp}_{\neg L}(K)} d\mu_{K}^{\rm MB} &\le \int_{{\rm proGrp}(K)\setminus {\rm proGrp}_{\neg L}(K)} \max_{\substack{E\le L\\E\ne K}}\#{\rm Epi}(\mathscr{G},(\Gal(E/K),\phi_{E/K}))d\mu_{K}^{\rm MB}\\
&\le \max_{\substack{E\le L\\E\ne K}}\int_{{\rm proGrp}(K)\setminus {\rm proGrp}_{\neg L}(K)} \#{\rm Epi}(\mathscr{G},(\Gal(E/K),S,\phi_{E/K,S}))d\mu_{K}^{\rm MB}\\
&\le \max_{\substack{E\le L\\E\ne K}}|\Gal(E/K)^{\rm ab}[|\mu(K)|]|^{-1}|\Gal(E/K)|^{-|S\cup P_\infty|+1}
\end{align*}
for any finite set of places $S$. Taking $S\to P_K$ proves the result, noting that $1 < |\Gal(E/K)| < \infty$.
\end{proof}

T\"urkelli posed a correction to Malle's prediction \cite{turkelli2015,alberts2021}, although this correction is a little ad hoc. Essentially, one partitions the count of $G$-extensions $L/K$ by the isomorphism class of $\mu(L)$ and apply the Malle-Bhargava principle to each case. There are only finitely many isomorphism classes that $\mu(L)$ can attain among $G$-extensions for a fixed finite group $G$, so the asymptotic growth rate is given by the maximum growth rate of the finitely many individual cases.

It would be interesting to refine the construction of $\mu_K^{\rm MB}$ in order to construct a new measure incorporating root of unity behavior, in such a way that ${\rm proGrp}_{\neg K(\zeta^e)}(K)$ is forced to be a null set, and see if this refined measure matches T\"urkelli's correction.

\subsection{The Grunwald problem}

The Grunwald problem $(G,K,S)$ is said to have a positive solution if the restriction map
\[
\Hom(G_K,G) \to \prod_{p\in S} \Hom(G_{K_p},G)
\]
is surjective. In other words, every local restriction for the places in $S$ is realizable in a global extension.

The Grunwald problem immediately generalizes to arbitrary groups with local data, simply by replacing $G_K$ with the group with local data in question $\mathscr{G}$. We can then ask for the probability that the Grunwald problem has a positive solution with respect to $\mu_K^{\rm MB}$.

\begin{theorem}\label{thm:Grunwald}
The event that $\mathscr{G}\in {\rm proGrp}(K)$ has a positive solution to the Grunwald problem $(G,\mathscr{G},S)$ for every $G$ and $S$ has measure $1$ with respect to $\mu_{K}^{\rm MB}$. That is,
\[
\mu_K^{\rm MB}\left(\left\{\mathscr{G}\in {\rm proGrp}(K) \Bigg{|} \substack{\displaystyle\text{for any finite group }G\text{ and finite set of places }S,\\\displaystyle \Hom(\mathscr{G},G) \to \prod_{p\in S} \Hom(G_{K_p},G)\text{ is surjective}}\right\}\right) = 1.
\]
\end{theorem}

Theorem \ref{thm:Grunwald} is not surprising, as $\mu_K^{\rm MB}$ was constructed from all possible local data. This is, however, distinct from the behavior of the absolute Galois group $G_K$. The Grunwald problem is known to have a negative solution for $(C_8,\Q,\{2\})$, following from the Grunwald-Wang theorem. The source of this failure is again roots of unity (specifically the $8^{\rm th}$ roots of unity for the $C_8$ case), and so it stands to reason that refining the measure to incorporate root of unity behavior might fix this issue as well.

Alternatively, we know the Grunwald problem has a positive solution for any set of finite places $S$ when $G$ is an odd solvable group or when $G$ has a generic extension. This includes odd order abelian groups and symmetric groups, for which we see $\mu_K^{\rm MB}$ predicting growth rates that agree with number field counting results (for fair counting functions). This suggests that similar results might be true for other groups for which the Grunwald problem always has a positive solution, detecting when we might reasonably expect Theorem \ref{thm:prob1intro} to make accurate predictions for number field counting.

\begin{proof}[Proof of Theorem \ref{thm:Grunwald}]
The proof follows from Theorem \ref{thm:muK_Malle_general}. For each $\phi\in \prod_{p\in S} \Hom(G_{K_p},G)$, take $\Sigma = (\Sigma_p)$ defined by
\[
\Sigma_p = \begin{cases}
\{\phi_p\} & p\in S\\
\Hom(G_{K_p},G) & p\not\in S.
\end{cases}
\]
Additionally, take the invariant $\inv$ to be given by a weight function in the following way:
\begin{itemize}
\item If there exists a $\Q$-conjugacy class (that is, a minimal set closed under conjugation and invertible powers) $\kappa\subseteq G$ for which $G\setminus \kappa$ is a subgroup of $G$, take the weight function
\[
w(g) = \begin{cases}
0 & g=1\\
1 & g\in \kappa\\
2 & g\in G\setminus (\kappa \cup \{1\}).
\end{cases}
\]
\item Otherwise, take the weight function
\[
w(g) = \begin{cases}
0 & g=1\\
1 & g\ne 1
\end{cases}
\]
giving the product of ramified primes ordering.
\end{itemize}
We claim that Theorem \ref{thm:muK_Malle_general}(iii) applies to each group under these orderings, proving that the $(G,\mathscr{G},S)$ Grunwald problem has a positive solution with probability $1$. Applying countable additivity to the countably many finite groups $G$ and finite sets of places $S$ then concludes the proof.

If there does not exist a $\Q$-conjugacy class (that is, a minimal set closed under conjugation and invertible powers) $\kappa\subseteq G$ for which $G\setminus \kappa$ is a subgroup of $G$, then every proper normal subgroup $N\normal G$ necessarily has $G\setminus N$ containing at least two $K$-conjugacy classes. Thus, the conditions of Theorem \ref{thm:muK_Malle_general}(iii) are automatically satisfied.

Alternatively, suppose such a $\kappa$ exists. We claim that no proper normal subgroup of $G$ can contain a minimal weight element. The conditions of Theorem \ref{thm:muK_Malle_general}(iii) would then be automatically satisfied. Consider that $G\setminus \kappa$ being a proper subgroup necessarily implies $|G\setminus \kappa| \le \frac{|G|}{2}$. Any normal subgroup $N\normal G$ that contains a minimal index element has nontrivial intersection with $\kappa$, and therefore contains $\kappa$ by virtue of being closed under conjugation and invertible powers. $N$ must also contain $1$, so we bound
\begin{align*}
|N| &\ge |\kappa| + 1\\
&=|G| - |G\setminus \kappa| + 1\\
&\ge |G| - \frac{|G|}{2} + 1\\
&>\frac{|G|}{2}.
\end{align*}
Thus, $[G:N] < 2$, which necessarily implies $N = G$ is not a proper subgroup.
\end{proof}

\subsection{Fair counting functions}

Wood proves that ordering abelian extensions by a fair counting function is sufficient for the leading constant to be given by a convergent Euler product \cite{wood2009}. This leading constant for counting abelian $G$-extensions of a number field $K$ agrees with Theorem \ref{thm:muK_Malle_general} if $G$ and $K$ are such that every Grunwald problem $(G,K,S)$ has a positive solution. Thus, away from the Grunwald problem obstruction, fairness of the ordering is sufficient for the growth rate for abelian groups in Theorem \ref{thm:muK_Malle_general} to agree with the true growth rate of Malle's counting function for abelian groups.

It natural to ask if the condition of fairness given by Wood can be weakened. Theorem \ref{thm:muK_Malle_general} gives an asymptotic growth rate for the counting function correspoinding to an admissible invariant $\inv$ counting surjections into a finite group $G$ in the following circumstance:
\begin{align}\label{eq:min_weight}
\text{$G$ is generated by the minimal weight elements with respect to $\inv$}
\end{align}

It is natural to ask if conditon (\ref{eq:min_weight}) is sufficient for the asymptotic growth rate in Theorem \ref{thm:muK_Malle_general} to agree with the true growth rate of Malle's counting function for counting $G$ extensions over a number field $K$ ordered by $\inv$. If this were true, it would indicate that the fairness obstruction is subsumed by the ``$G$ not generated by minimal index elements'' obstruction.

However, this is known to be false for counting $G$-extensions. Condition (\ref{eq:min_weight}) is not sufficient for the asymptotic growth rate in Theorem \ref{thm:muK_Malle_general} to agree with the true growth rate of Malle's counting function. Wood shared a counter example with the author in a personal correspondence \cite{wood_note_fair}: Order $\Z/4\Z$-extensions of $\Q(\sqrt{2})$ by the admissible invariant $\inv$ determined by the weight function
\[
w(g) = \begin{cases}
0 & g=0\\
1 & g=1,3\\
2 & g=2.
\end{cases}
\]
It is clear that $\Z/4\Z$ is generated by minimal weight elements, and so satisfies condition (\ref{eq:min_weight}). However, explicit computation (which Wood does following along Wright's original proof of Malle's conjecture for abelian extensions \cite{wright1989,wood_note_fair}) reveals that
\[
\#\{K/\Q(\sqrt{2}) : \Gal(K/\Q(\sqrt{2}))\cong \Z/4\Z, |\inv(K/\Q(\sqrt{2}))| \le X\} \sim c X
\]
for a positive constant $c$ given explicitly by the sum of \emph{two distinct} convergent Euler products. This is certainly different than the growth rate given in Theorem \ref{thm:muK_Malle_general}, which involves only a single convergent Euler product.

Wood's counter example shows that there is some nontrivial condition beyond (\ref{eq:min_weight}) an ordering must satisfy for Theorem \ref{thm:muK_Malle_general} to match the true growth rate of Malle's counting function. Fairness is sufficient in the case of abelian groups avoiding the Grunwald problem obstruction, and it would be an interesting question to determine if fairness is also necessary in such cases.

Moving beyond abelian groups, the notion of a fair counting function defined in \cite{wood2009} does not have an immediate generalization to nonabelian groups. To the author's knowledge there has been no proposed extension of the notion of a fair counting function to nonabelian extensions, other than the widespread recognition that the product of ramified primes ordering should be considered fair. Possibly, the discriminant ordering for $S_n$-extensions in degree $n$ should also be considered fair for similar reasons \cite{bhargava2007}. This makes it difficult to determine for which orderings we should expect the growth rates in Theorem \ref{thm:muK_Malle_general} to agree with the true growth rate of Malle's counting functions (even assuming all other obstructions can be avoided).

Vaguely speaking, the source of the definition of ``fair" counting functions for abelian extensions is similar to that of the Grunwald problem. In the generating Dirichlet series computations in \cite{wood2009}, there occurs a finite sum of Euler products with rightmost pole of maximal order. When the sum cancels out, it can be proven that the entire generating Dirichlet series cancels out giving a negative solution to the corresponding Grunwald problem. When the sum does not cancel out, the individual Euler products may produce different leading constants that must be added together. A fair counting function forces this sum to only include those Euler products which produce the same leading constant. Because of this relationship, the author is optimistic that incorporating the Grunwald problem into the random model will shed light on what a fair counting function should be for nonabelian extensions. In particular, a good next step in this direction is to refine the measure to include root of unity behavior and ask which orderings have nice local probabilities in analogy with Wood's main results in \cite{wood2009}.

\bibliographystyle{alpha}
\bibliography{A_Random_Group_With_Local_Data.bbl}
\end{document}